\documentclass[psamsfonts]{amsart}
\usepackage{amssymb}
\usepackage{amscd}
\usepackage{mathtools}
\usepackage{amscd, enumerate}
\usepackage{tikz-cd}
\usepackage[utf8]{inputenc}
\usepackage[all]{xy}
\usepackage{hyperref}

\newtheorem{thm}{Theorem}[section]
\newtheorem{prop}[thm]{Proposition}
\newtheorem{lem}[thm]{Lemma}
\newtheorem{cor}[thm]{Corollary}

\theoremstyle{definition}
\newtheorem{fact}[thm]{Fact}
\newtheorem{definition}[thm]{Definition}
\newtheorem{construction}[thm]{Construction}
\newtheorem{notation}[thm]{Notation}

\newtheorem{ex}[thm]{Example}
\newtheorem{rem}[thm]{Remark}

\theoremstyle{remark}

\newtheorem{nota}[thm]{Notation}

\DeclareMathOperator*{\ModR}{Mod-R}

\DeclareMathOperator*{\RMod}{R-Mod}
\DeclareMathOperator*{\modR}{mod-R}
\DeclareMathOperator{\ext}{Ext}
\DeclareMathOperator{\Tel}{Tel}

\DeclareMathOperator*{\im}{Im}
\DeclareMathOperator{\tor}{Tor}
\DeclareMathOperator{\Hom}{Hom}
\DeclareMathOperator{\End}{End}
\DeclareMathOperator*{\Ann}{Ann}
\DeclareMathOperator*{\Ker}{Ker}
\DeclareMathOperator*{\Coker}{Coker}
\DeclareMathOperator*{\spec}{Spec}

\DeclareMathOperator{\cogen}{Cogen}

\DeclareMathOperator*{\ass}{Ass}
\DeclareMathOperator*{\Span}{Span}
\DeclareMathOperator*{\Soc}{Soc}
\DeclareMathOperator*{\vass}{VAss}
\DeclareMathOperator*{\pd}{pd}
\DeclareMathOperator*{\id}{id}
\DeclareMathOperator*{\tr}{Tr}

\DeclareMathOperator*{\Supp}{Supp}

\DeclareMathOperator*{\Ch}{\mathbf{C}}
\DeclareMathOperator*{\Der}{\mathbf{D}}
\DeclareMathOperator*{\Loc}{Loc}
\DeclareMathOperator*{\Cone}{Cone}

\begin{document}
\title{Tilting classes over commutative rings}

\author{Michal Hrbek}
\email{hrbmich@gmail.com}
\address{Charles University \\
   Faculty of Mathematics and Physics\\
   Department of Algebra \\
   Sokolovsk\'a 83\\
   186 75 Praha 8\\
   Czech Republic}

\author{Jan \v{S}\v{t}ov\'{\i}\v{c}ek}
\email{stovicek@karlin.mff.cuni.cz}
\address{Charles University \\
   Faculty of Mathematics and Physics\\
   Department of Algebra \\
   Sokolovsk\'a 83\\
   186 75 Praha 8\\
   Czech Republic}

\thanks{Both authors were partially supported by the Grant Agency of the Czech Republic under the grant no.~14-15479S. The first named author was also partially supported by the project SVV-2016-260336 of the Charles University in Prague.} 


\begin{abstract}
		We classify all tilting classes over an arbitrary commutative ring via certain sequences of Thomason subsets of the spectrum, generalizing the classification for noetherian commutative rings by Angeleri-Pospíšil-Šťovíček-Trlifaj (\cite{APST}). We show that the $n$-tilting classes can be equivalently expressed as classes of all modules vanishing in the first $n$ degrees of one of the following homology theories arising from a finitely generated ideal: $\tor_*(R/I,-)$, Koszul homology, \v{C}ech homology, or local homology (even though in general none of those theories coincide). Cofinite-type $n$-cotilting classes are described by vanishing of the corresponding cohomology theories. For any cotilting class of cofinite type, we also construct a corresponding cotilting module, generalizing the construction of Šťovíček-Trlifaj-Herbera (\cite{HST}). Finally, we characterize cotilting classes of cofinite type amongst the general ones, and construct new examples of $n$-cotilting classes \emph{not} of cofinite type, which are in a sense hard to tell apart from those of cofinite type.
\end{abstract}

\maketitle

\setcounter{tocdepth}{1} \tableofcontents

\section{Introduction}

Infinitely generated tilting and cotilting modules were introduced in \cite{CT,AC} about two decades ago as a formal generalization of Miyashita tilting modules~\cite{Miy}. The main motivation for studying Miyashita tilting modules is that they represent equivalences of derived categories. In the last few years, it became clear that infinitely generated modules represent derived equivalences as well (see~\cite{Baz,BMT,Sto,FMS,PS}), but also that there is more than that.

In the realm of commutative noetherian rings, it was shown~\cite{APST} that tilting modules have a very close relation to the underlying geometry of the corresponding affine schemes. In fact, a full classification of tilting modules up to additive equivalence in terms of geometric data was obtained there. From a wider perspective, this was explained by Angeleri and Saor\'{\i}n~\cite{AHS} who exhibited a direct relation of the dual cotilting modules to $t$-structures of the derived category, and the resulting $t$-structures were known to have a similar classification from~\cite{AJS}. The outcome was that in the commutative noetherian setting, cotilting modules represent a class of $t$-structures (namely those, which are compactly generated and induce a derived equivalence to the heart).

The motivating question of the present work is, what remains true \emph{for commutative rings which are not necessarily noetherian}. A posteriori, the general setting forced us to look for more conceptual proofs and gave more insight in the problem even for noetherian rings. Here is a short overview of highlights of the paper:

\begin{enumerate}
 \item For a commutative ring $R$, we give a full classification of tilting $R$-modules up to additive equivalence in terms of filtrations of Thomason subsets of $\spec R$ (Theorem~\ref{T:mainthm2}).
 \item We obtain more insight in $R$-modules which are perfect in the derived category as well as in resolving subcategories formed by such modules, in that every such resolving subcategory is generated by syzygies of Koszul complexes (this follows from Theorem~\ref{T:mainthm2}(iv) combined with Proposition~\ref{P14}).
\item We observe that tilting modules up to additive equivalence are also classified by vanishing of the derived completion functors as well as by vanishing of the \v{C}ech homology (Theorem~\ref{T:mainthm3}).
\item We give a construction of dual cotilting modules to tilting modules (Theorem~\ref{thmconstruction}) and present some intriguing examples in the last section.
\end{enumerate}

A starting point for the classification in this paper (as well as for its predecessor~\cite{H} which focused on (co)tilting modules of homological dimensions at most one) is the 1997 paper of Thomason~\cite{Tho}. There he  generalized Neeman-Hopkins's classification of thick subcategories of the derived category of perfect complexes over a commutative noetherian ring to a general commutative ring.
The parametrizing family for this classification is the set of all sets of the spectrum of the ring, which are open with respect to the Hochster dual of the usual Zariski topology. If the ring is noetherian, these \emph{Thomason} sets coincide precisely with the specialization closed subsets (that is, upper subsets of $\spec R$ with respect to set-theoretic inclusion). It turns out that the main classification theorem for tilting modules and classes induced by them in~\cite{APST} (but not its proof) remains valid in the non-noetherian setting provided that we simply change `specialization closed set' to 'Thomason set' in the statement.

As was the case in the noetherian case, it is convenient to carry out the main steps of the proof in the dual setting of cotilting modules and cotilting classes induced by them first, and then transfer the results via elementary duality. The reason why the dual setting is more graspable seems to lie in the fact that, over a commutative ring, the cotilting classes dual to a tilting class are closed under injective envelopes (Proposition~\ref{P01}). In homological dimension one, this already leads to the well-understood theory of hereditary torsion pairs, as explained in \cite{H}. We do not know of any analogous closure property for tilting classes. 

In other respects, our approach differs considerably from \cite{APST}. For example, we cannot use the Matlis theory of injectives, the classical theory of associated primes or the Auslander-Bridger transpose anymore. Instead, we use the classification of hereditary torsion pairs of finite type from~\cite{H} and prove directly that a cotilting class is described cosyzygy-wise by a sequence of such torsion pairs (Lemma~\ref{L02}).

Another problem we have to tackle is that tilting classes bijectively correspond to resolving subcategories of modules having finite projective resolutions by finitely generated projectives (in the sense of Theorem~\ref{T:finitetype}). The description of such resolving subcategories in the noetherian case was obtained, independently of~\cite{APST}, by Dao and Takahashi \cite{DT}. It turns out that a sufficiently rich supply of $R$-modules with such a finite resolution, for any commutative ring $R$, comes from tails of classical Koszul complexes associated with finite sets of elements of $R$. The key point here is to understand when exactly such tails are exact, i.e.\ when high enough Koszul homologies of $R$ vanish. 

Apart from the homological description of tilting classes which was obtained for noetherian rings in \cite{APST}, i.e.\ in terms of vanishing of certain degrees of $\tor_*(R/I,-)$, we also obtain another one in terms of vanishing of the Koszul homology. We proceed further here to show that two other homology theories also fit in this place---the local and the \v{C}ech homology associated to a finitely generated ideal $I$. The interesting point is that, unless the ring is noetherian, these two homologies need not be isomorphic. However, the vanishing of the first $n$ homologies of a module is always equivalent for all of the four homology theories in play. The analogous result holds for the dual setting too for more classical local and \v{C}ech cohomologies.

\smallskip

The paper is organized as follows. Section~\ref{sec:torsion-comm} gathers the required results about hereditary torsion pairs of commutative rings, focusing on those of finite type. In Section~\ref{sec:grade} we prove the key Proposition~\ref{P:Koszul-ext} which shows that vanishing of $\ext_R^i(R/I,-)$ for $i=0,1,\ldots,n-1$ is equivalent to vanishing of the corresponding Koszul cohomology for any commutative ring. We link this vanishing property with the notion of vaguely associated prime ideal, yielding results which are analogous to characterizations of the grade of a finitely generated module over a noetherian ring. The fourth section recalls the fundamentals of the theory of large tilting modules over an arbitrary ring. These three sections prepare the ground for the core Section~\ref{sec:cotilt}, in which we classify the $n$-cotilting classes of cofinite type over a commutative ring. These results are then reformulated and translated for the tilting side of the story in Section~\ref{sec:main-results}. The aforementioned connection with the derived functors of torsion and completion, as well as the \v{C}ech (co)homology, is explained in Section~\ref{sec:derived}. In the following Section~\ref{sec:construct}, we show how a cotilting module associated to any cotilting class of cofinite type over a commutative ring can be constructed, building on the idea from \cite{HST}. In the final Section~\ref{sec:notcofinite}, we characterize the cofinite-type cotilting classes amongst the general ones, and we construct new examples of $n$-cotilting classes which are not of cofinite type, but which are in a sense difficult to tell apart from cofinite type cotilting classes.
	
\section{Torsion pairs over commutative rings}
\label{sec:torsion-comm}

In this section we give a classification of hereditary torsion pairs of finite type over commutative rings by Thomason subsets of the Zariski spectrum. This will be a key tool further in the paper. The material is probably known or not difficult to see for experts and almost all the fragments of the classification are present in the literature, but we have not been able to find a convenient reference.

Regarding our notation, $R$ will always stand for an associative and unital ring and $\ModR$ for the category of right $R$-modules. We will always assume that $R$ is also commutative unless stated otherwise. If $I \subseteq R$ is an ideal, we denote by
\[ V(I) = \{ \mathfrak{p} \in \spec R \mid I \subseteq \mathfrak{p} \} \]
the corresponding Zariski closed set. We begin our discussion with so-called Hochster duality of spectral topological spaces and Thomason sets.

\subsection{Spectral spaces, Hochster duality, and Thomason sets}

The set of Thomason subsets of $\spec R$ was used as an index set in Thomason's classification~\cite{Tho} of thick subcategories of the category of perfect complexes. Although Thomason sets already appear in the work of Hochster~\cite{Hoch}, their name is customary nowadays since they were used in connection with other classification problems (see e.g.~\cite{KP} and the references therein) and, in particular, with tilting theory in~\cite{H}. Let us start the discussion with a definition, which is relatively simple:

\begin{definition} \label{D:Thomason}
Given a commutative ring $R$, a subset $X \subseteq \spec R$ is a \emph{Thomason set} if it can be expressed as a union of complements of quasi-compact Zariski open sets. That is, there exists a collection $\mathcal{G}$ of finitely generated ideals of $R$ such that
		\[ X = \bigcup_{I \in \mathcal{G}} V(I), \]
where $U_I = \spec R \setminus V(I)$ are the quasi-compact Zariski open sets.
\end{definition}

To put this into the right context, one should note that Thomason sets define a topology on $\spec R$ for which they are the open sets. This observation comes from Hochster's~\cite[Proposition 8]{Hoch} and, following \cite[\S1]{KP}, it can be also explained as follows. First, $\spec R$ with the Zariski topology is a so-called spectral space.

\begin{definition}
A topological space $X$ is \emph{spectral} (or \emph{coherent} in the terminology of~\cite[\S II.3.4, p. 65]{Jo}) if
\begin{enumerate}
 \item the class of quasi-compact open sets is closed under intersections and forms an open basis for the topology, and
 \item every irreducible closed set is the closure of a unique point.
\end{enumerate}
\end{definition}

In fact, Hochster~\cite{Hoch} proved that spectral spaces are up to homeomorphism precisely the ones of the form $\spec R$ for a commutative ring $R$. If $X$ is a spectral topological space, the collection of quasi-compact open sets with the operations of the set-theoretic union and intersection is a distributive lattice. In fact, $X$ is fully determined by this distributive lattice and every distributive lattice arises like that. Recall that an ideal $I$ in a lattice $L$ is a prime ideal if $x \wedge y \in I$ implies $x \in I$ or $y \in I$.

\begin{prop}[{\cite[Corollaries II.1.7(i) and II.3.4]{Jo}}] \label{P:distrib}
There is a bijective correspondence between
\begin{enumerate}
 \item homeomorphism classes of spectral spaces $X$, and
 \item isomorphism classes of distributive lattices $L$,
\end{enumerate}
given as follows. To a space $X$ we assign the lattice of quasi-compact open sets. On the other hand, given $L$, the underlying set of $X$ is the set of all lattice prime ideals $\mathfrak{p} \subseteq L$, and the open basis of the topology is formed by the quasi-compact sets $U_x = \{ \mathfrak{p} \in X \mid x \not\in \mathfrak{p} \}$, $x \in L$ (this is a lattice theoretic version of the Zariski topology).
\end{prop}

\begin{rem} \label{R:open-sets}
The lattice $\Omega(X)$ of all open sets in a topological space $X$ is always a distributive lattice. If $X$ is spectral and $L \subseteq \Omega(X)$ is the sublattice of all quasi-compact open sets, then $\Omega(X)$ is isomorphic to the lattice of all ideals $L$ (see the proof \cite[Proposition II.3.2]{Jo}).
\end{rem}

Now, the opposite lattice of $L^\mathrm{op}$ of a distributive lattice $L$ is again distributive. Since the complement of a prime ideal in $L$ is a prime filter in $L$ and hence a prime ideal in $L^\mathrm{op}$ (see~\cite[Prop. I.2.2]{Jo}), the space corresponding to $L^\mathrm{op}$ has (up to this canonical identification) the same underlying set as the one corresponding to $L$. The topology is different, however. Starting with a spectral space $X$, the dual topology will have as an open basis precisely the complements of the original quasi-compact open sets. The resulting space is called the \emph{Hochster dual} of $X$.

If $X = \spec R$ with the Zariski topology, the open sets in the dual topology are precisely the Thomason sets. The following immediate consequence of the present discussion will be useful to us.

\begin{cor} \label{C:dual}
Let $R$ be a commutative ring and consider $X = \spec R$ with the Zariski topology. Then there is a bijective correspondence between
\begin{enumerate}
\item Thomason subsets of $\spec R$, and
\item filters in the lattice of quasi-compact open sets of $X$.
\end{enumerate}
Given a Thomason set $X \subseteq \spec R$, we assign to it the filter of quasi-compact open sets $\{U \mid \spec R \setminus U \subseteq X \}$. Conversely, given a filter $\mathcal{F}$, we assign to it $X = \bigcup_{U \in \mathcal{F}} (\spec R \setminus U)$. 
\end{cor}

\subsection{Hereditary torsion pairs of finite type}

A {torsion pair} in $\ModR$ is a pair of classes $(\mathcal{T},\mathcal{F})$ of modules such that $\Hom_R(T,F) = 0$ for each $T \in \mathcal{T}$ and $F \in \mathcal{F}$ and for each $X \in \ModR$ there exists a short exact sequence $0 \to T \to X \to F \to 0$. This short exact sequence is unique up to a unique isomorphism. The class $\mathcal{T}$ is called a \emph{torsion class} and torsion classes of a torsion pair $\mathcal{T}$ are characterized by the closure properties---they are closed under coproducts, extensions and factor modules. Dually, \emph{torsion-free classes} $\mathcal{F}$ are characterized as those being closed under products, extensions and submodules.

A torsion pair is \emph{hereditary} if $\mathcal{T}$ is closed under submodules. Equivalently, $\mathcal{F}$ is closed under injective envelopes, \cite[\S VI.3]{St}.

We will be mostly interested in the case when the torsion class is generated by a set of finitely presented modules. That is, $\mathcal{F} = \Ker\Hom_R(\mathcal{T}_0,-)$ for a set of finitely presented modules $\mathcal{T}_0$. Although all the concepts can be defined over any ring, commutative or not, the following is special in the commutative case.

\begin{prop} \label{P:fin-to-hered}
Let $R$ be a commutative ring and $(\mathcal{T},\mathcal{F})$ be a torsion pair such that $\mathcal{F} = \Ker\Hom_R(\mathcal{T}_0,-)$ for a set of finitely presented modules $\mathcal{T}_0$. Then $(\mathcal{T},\mathcal{F})$ is hereditary.
\end{prop}

\begin{proof}
We refer to (the proof of) \cite[Lemma 4.2]{AH}.
\end{proof}

A hereditary torsion pair is determined by the class of cyclic torsion modules, or equivalently by the set of ideals $I$ for which $R/I$ is torsion, \cite[Proposition VI.3.6]{St}. Such a set of ideals corresponding to a hereditary torsion pair is called a \emph{Gabriel topology}. We refer to~\cite[\S VI.5]{St} for the general definition of a Gabriel topology, while a much simpler special case will be discussed below in Lemma~\ref{L:Gabriel-finite}.

\begin{prop} \label{P:torsion}
Let $R$ be a commutative ring. There is a bijective correspondence between hereditary torsion pairs $(\mathcal{T},\mathcal{F})$ in $\ModR$ and Gabriel topologies $\mathcal{G}$ on $R$.
Moreover, the following are equivalent for a hereditary torsion pair $(\mathcal{T},\mathcal{F})$:
\begin{enumerate}
\item There is a set $\mathcal{T}_0$ of finitely presented modules with $\mathcal{F} = \Ker\Hom_R(\mathcal{T}_0,-)$.
\item $\mathcal{F}$ is closed under direct limits.
\item The corresponding Gabriel topology $\mathcal{G}$ has a basis of finitely generated ideals (i.e.\ each ideal in $\mathcal{G}$ contains a finitely generated ideal in $\mathcal{G}$).
\end{enumerate}
\end{prop}

\begin{proof}
We refer to \cite[\S VI.6]{St} for the bijective correspondence. The second part has been proved in \cite[Lemma 2.4]{H}. 
\end{proof}

\begin{definition} \label{D:finite-type}
The torsion pairs satisfying the conditions (1)--(3) above will be called hereditary torsion pairs \emph{of finite type}.
\end{definition}

If $\mathcal{G}$ is a Gabriel topology on $R$, we denote by $\mathcal{G}^f$ the collection of all finitely generated ideals in $\mathcal{G}$. It is always closed under products of ideals (i.e.\ if $I_1,I_2 \in \mathcal{G}$, then $I_1\!\cdot\!I_2 \in \mathcal{G}$) and finitely generated overideals. If Propositions~\ref{P:torsion}(3) holds, $\mathcal{G}^f$ completely determines $\mathcal{G}$ and the two latter closure properties in fact completely characterize such Gabriel topologies.

\begin{lem}[{\cite[Lemma 2.3]{H}}] \label{L:Gabriel-finite}
Let $R$ be a commutative ring. Then a filter $\mathcal{G}$ of ideals of $R$ with a basis of finitely generated ideals is a Gabriel topology if and only if it is closed under products of ideals.
\end{lem}

The following correspondence is a consequence of standard commutative algebra.

\begin{lem} \label{L:to-lattice}
For a commutative ring $R$, there is a bijective correspondence between
\begin{enumerate}
\item Gabriel topologies with a basis of finitely generated ideals, and
\item filters of the lattice of quasi-compact Zariski open subsets of $\spec R$.
\end{enumerate}
\end{lem}

\begin{proof}
Quasi-compact Zariski open sets are precisely those of the form $U_I = \spec R \setminus V(I)$ for a finitely generated ideal $I \subseteq R$. Moreover, $U_I \subseteq U_{I'}$ if and only if $V(I) \supseteq V(I')$ if and only if
$$ I \subseteq \bigcap_{\mathfrak{p}\supseteq I'} \mathfrak{p} = \sqrt{I'} $$
if and only if $I^n \subseteq I'$ for some $n \ge 1$. Since also $U_{I_1\!\cdot\!I_2} = U_{I_1} \cap U_{I_2}$, the correspondence which assigns to a Gabriel topology $\mathcal{G}$ as in (1) the filter $\{ U_I \mid I \in \mathcal{G} \}$ of the lattice of quasi-compact open sets is a bijective.
\end{proof}

If we combine the discussion above with Corollary~\ref{C:dual}, we obtain the following pa\-ra\-met\-ri\-za\-tion of hereditary torsion pairs of finite type (see \cite{AH,H} for closely related results and compare also to \cite[Proposition VI.6.15]{St}). Here, if $M\in\ModR$, $\Supp(M) = \{ \mathfrak{p} \in \spec R \mid M_\mathfrak{p} \ne 0 \}$ denotes as usual the support of $M$.

\begin{prop} \label{P:torsion-Thomason}
Given a commutative ring $R$, there is a bijective correspondence between
\begin{enumerate}
\item hereditary torsion pairs $(\mathcal{T},\mathcal{F})$ of finite type, and
\item Thomason subsets of $\spec R$.
\end{enumerate}
Given a torsion pair $(\mathcal{T},\mathcal{F})$, we assign to it the subset $X = \bigcup_{T \in \mathcal T} \Supp(T)$. Conversely, if $X$ is a Thomason set, we put $\mathcal{T} = \{ T \in \ModR \mid \Supp(T) \subseteq X \}$.
\end{prop}

\begin{proof}
The fact that there is a bijective correspondence between (1) and (2) follows by combining Proposition~\ref{P:torsion}, Lemma~\ref{L:to-lattice} and Corollary~\ref{C:dual}. The particular correspondence given by the three statements is rather explicit. To $(\mathcal{T},\mathcal{F})$ we assign the Thomason set $X$ of all prime ideals $\mathfrak{p} \in \spec R$ for which there exists a finitely generated ideal $I \subseteq \mathfrak{p}$ with $R/I \in \mathcal{T}$. Conversely, to a Thomason set $X$ we assign the unique hereditary torsion pair of finite type such that, given a finitely generated ideal $I$, $R/I$ is torsion precisely when $V(I) \subseteq X$.

Suppose now that $(\mathcal{T},\mathcal{F})$ is of finite type and $X$ is the corresponding Thomason set. Since every object in $\mathcal{T}$ is an epimorphic image of a coproduct of cyclic modules contained in $\mathcal{T}$, we have $X = \bigcup_{R/I \in \mathcal{T}} \Supp(R/I) = \bigcup_{T \in \mathcal T} \Supp(T)$, as required. Next let $\mathcal{T}' = \{ T \in \ModR \mid \Supp(T) \subseteq X \}$. This is clearly a hereditary torsion class and, given a finitely generated ideal $I \subseteq R$, we have $R/I \in \mathcal{T'}$ if and only if $\Supp R/I \subseteq X$. As $\Supp R/I = V(I)$ in this case, $\mathcal{T'}$ is also sent to $X$ under the bijective correspondence, and hence $\mathcal{T} = \mathcal{T'}$.
\end{proof}

We conclude the discussion with a description of the torsion-free classes under the correspondence from Proposition~\ref{P:torsion-Thomason}. This was important in the classification of tilting classes in the noetherian case~\cite{APST} and the current version comes from~\cite{H}.

\begin{definition}[{\cite[\S3.2]{H}}] \label{D:vaguely-ass}
Let $M$ be a module over a commutative ring. A prime ideal $\mathfrak{p}$ is \emph{vaguely associated} to $M$ if the smallest class of modules containing $M$ and closed under submodules and direct limits contains $R/\mathfrak{p}$. The set of primes vaguely associated to $M$ is denoted by $\vass(M)$.
\end{definition}

If $R$ is in addition noetherian, $\vass(M)$ coincides with the set of usual associated prime ideals by~\cite[Lemma 3.8]{H}. The following proposition then generalizes~\cite[Proposition 2.3(iii)]{APST}.

\begin{prop} \label{P:torsion-free-Thomason}
Let $R$ be a commutative ring and $(\mathcal{T},\mathcal{F})$ a hereditary torsion pair of finite type. If $X$ is the Thomason set assigned to the torsion pair by Proposition~\ref{P:torsion-Thomason}, then
$$ \mathcal{F} = \{ F \in \ModR \mid \vass(M) \cap X = \emptyset \}. $$
\end{prop}

\begin{proof}
This has been proved in~\cite{H}. Namely, given $\mathfrak{p} \in \spec R$ and a finitely generated $R$-module $M$, then $\Hom_R(M,R/\mathfrak{p})\ne 0$ if and only if $\mathfrak{p}\in\Supp(M)$ by \cite[Lemma 3.10]{H}. Then necessarily $\vass(M) \cap X = \emptyset$ for each $F \in \mathcal{F}$. Conversely, suppose that $F \not\in \mathcal{F}$. Then there is an embedding $i\colon R/J \to F$ with $R/J \in \mathcal{T}$. By the proof of \cite[Lemma 3.2]{H} there exists $\mathfrak{p} \in \vass(M)$ such that $\mathfrak{p}\supseteq I$. That is, $\mathfrak{p} \in \vass(M) \cap X$.
\end{proof}

\section{Generalized grade of a module}
\label{sec:grade}
A very important concept in homological algebra for modules over commutative noetherian rings is the one of a regular sequence and of the grade of a module. The maximal length of a regular sequence in a given ideal has various homological characterizations. Appropriate forms of these characterizations still remain equivalent over non-noetherian commutative rings which will be useful for us. We shall give details in this section.

We shall use the following notation here. Given $M \in \ModR$ and $i \geq 0$, we denote by $\Omega^i(M)$ the $i$-th syzygy of $M$ (uniquely determined only up to projective equivalence). If $i < 0$, we let $\Omega^{i} M$ stand for the minimal $\lvert i \rvert$-th cosyzygy of $M$ (determined uniquely as a module).

\subsection{Derived categories and truncation of complexes}
\label{subsec:derived}

In this section it will also be useful to argue using the derived category $\Der(\ModR)$ of $\ModR$; see for instance~\cite[Chapter 13]{KS}. The suspension functor will be denoted by
$$ \Sigma\colon \Der(\ModR) \to \Der(\ModR) $$
and we will typically use the homological indexing of components of complexes:
$$ X\colon \quad \cdots \to X_{n+1} \xrightarrow{d_{n+1}} X_{n} \xrightarrow{d_{n}} X_{n-1} \xrightarrow{d_{n-1}} X_{n-2} \to \cdots $$

In this context, we will also use the homological truncation functors with respect to the (suspensions of the) canonical $t$-structure on $\Der(\ModR)$ (see~\cite[Examples 1.3.22]{BBD} or~\cite[\S12.3]{KS}). We shall slightly adapt our notation for it to be compatible with our homological indexing. Thus, given a complex $X \in \Der(\ModR)$ and $n \in \mathbb{Z}$, we shall denote by $X_{\ge n}$ the complex
$$ X_{\ge n}\colon \quad \cdots \to X_{n+1} \xrightarrow{d_{n+1}} X_{n} \xrightarrow{d_{n}} \im d_{n} \to 0 \to \cdots $$

The subcomplex inclusion $i\colon X_{\ge n}\colon \to X$ (when we consider the complexes in the usual category of complexes $\Ch(\ModR)$) induces an isomorphism on the $k$-th homology for each $k \ge n$ and $H_k(X_{\ge n}) = 0$ for all $k<n$. Dually, the projection morphism $p\colon X \to X/X_{\ge n}$ induces an isomorphism on the $k$-th homology for all $k<n$ and $H_k(X/X_{\ge n}) = 0$ whenever $k \ge n$.

In fact, $(X_{\ge n} \mid n \in \mathbb{Z})$ yields a filtration of $X$ in $\Ch(\ModR)$. Since the homologies of $X_{\ge n}/X_{\ge n+1}$ in degrees different from $n$ vanish, we have for each $n \in \mathbb{Z}$ a distinguished triangle
\begin{equation}
X_{\ge n+1} \to X_{\ge n} \to \Sigma^{n} H_n(X) \to \Sigma X_{\ge n+1}.
\label{E:1-step-filt}
\end{equation}
in $\Der(\ModR)$. This observation, which formalizes how a complex can be built by extension from its homology modules will be especially useful in~\S\ref{sec:cotilt}.

\subsection{Koszul complexes and Koszul (co)homology}
\label{subsec:Koszul}

A particularly useful class of complexes here will be the so-called Koszul complexes. Here $R$ will stand for a commutative ring.

\begin{definition}[{\cite[\S1.6]{BH}, \cite[\S8.2]{N}}]
Given $x \in R$, we define the \emph{Koszul complex} with respect to $x$, denoted $K_\bullet(x)$, as follows
$$0 \rightarrow R \xrightarrow{-\cdot x} R \rightarrow 0,$$
by which we mean a complex concentrated in degree 1 and 0, with the only non-zero map $R \rightarrow R$ being the multiplication by $x$. Here we use the homological indexing (i.e.\ the differential has degree $-1$).

More generally, given a finite sequence $\mathbf{x}=(x_1,x_2,\ldots,x_n)$ of elements of $R$, we define the complex $K_\bullet(\mathbf{x})$ as the tensor product $\bigotimes_{i=1}^n K_\bullet(x_i)$ of complexes of $R$-modules.
\end{definition}

In particular, $K_\bullet(\mathbf{x})$ is a complex of finitely generated free $R$-modules concentrated in degrees $0$ to $n$. Given an arbitrary module, we can define the corresponding Koszul homology and cohomology.

\begin{definition}
Given a finite sequence $\mathbf{x}=(x_1,x_2,\ldots,x_n)$ of elements of $R$, a module $M\in\ModR$, and $i \in \mathbb{Z}$, the \emph{$i$-th Koszul homology} of $\mathbf{x}$ with coefficients in $M$ is defined as
$$
H_i(\mathbf{x}; M) = H_i\big(K_\bullet(\mathbf{x}) \otimes_R M\big).
$$
Similarly, the \emph{$i$-th Koszul cohomology} of $\mathbf{x}$ with coefficients in $M$ is defined as
$$
H^i(\mathbf{x}; M) = H^i\big(\Hom_R(K_\bullet(\mathbf{x}), M)\big).
$$
\end{definition}

\begin{rem}
The Koszul cohomology has a particularly easy interpretation in the derived category $\Der(\ModR)$. Namely, we have
$$ H^i(\mathbf{x}; M) \cong \Hom_{\Der(\ModR)}(K_\bullet(\mathbf{x}), \Sigma^i M). $$
\end{rem}

Typically we will start with a finitely generated ideal $I = (x_1, \dots, x_n)$ of $R$ and we will consider the Koszul complex or homology or cohomology of $\mathbf{x}=(x_1,x_2,\ldots,x_n)$. These notions are \emph{not} invariant under change of the generating set of $I$---see~\cite[Proposition 1.6.21]{BH} for a precise discussion on how the complex changes. However, the following consequence of this discussion will be crucial for us and will make the particular choice of a finite generating set of $I$ unimportant.

\begin{prop}
Let $R$ be a commutative ring, $I$ be a finitely generated ideal and let $\mathbf{x}=(x_1,x_2,\ldots,x_n)$ and $\mathbf{y}=(y_1,y_2,\ldots,y_m)$ be two systems of generators of $I$. Given any integer $\ell \in \mathbb{Z}$, one has $H^i(\mathbf{x}; M) = 0$ for all $i \le \ell$ if and only if $H^i(\mathbf{y}; M) = 0$ for all $i \le \ell$.
\end{prop}

\begin{proof}
This is an immediate consequence of \cite[Corollary 1.6.22]{BH} and \cite[Corollary 1.6.10(d)]{BH}.
\end{proof}

Hence, we introduce the following slightly abused notation.

\begin{nota} \label{N:Koszul}
Given a commutative ring $R$ and a finitely generated ideal $I$, we always fix once and for ever a system of generators $\mathbf{x}=(x_1,x_2,\ldots,x_n)$. We then say that $K_\bullet(\mathbf{x})$ is the \emph{Koszul complex} of the ideal $I$ and denote it by $K_\bullet(I)$.

Similarly for $M \in \ModR$ and $i \in \mathbb{Z}$ we define the \emph{$i$-th Koszul homology} of $I$ with coefficients in $M$ as $H_i(I;M) = H_i(\mathbf{x};M)$, and similarly for the \emph{$i$-the Koszul cohomology} $H^i(I;M) = H^i(\mathbf{x};M)$.
\end{nota}

The following two well-known properties of Koszul complexes which will be important in our application.

\begin{lem} \label{L:Koszul-basic}
  Let $I$ be a finitely generated ideal. Then:
  \begin{enumerate}
    \item[(i)] $H_0(K_\bullet(I)) = R/I$,
    \item[(ii)] $I \subseteq \Ann\big(H_j(K_\bullet(I)\big)$ for all $k=0,1,\ldots,n$.
  \end{enumerate}
\end{lem}
\begin{proof}
  See \cite[p. 360, (8.2.7) and p. 364, Theorem 3]{N} or \cite[Prop. 1.6.5(b)]{BH}.
\end{proof}

In fact, the latter lemma allows us to relate the Koszul cohomology of $I$ to ordinary $\ext$-groups. Namely, we have a map of complexes
\begin{equation}
\label{E:Koszul-derived}
q\colon K_\bullet(I) \to K_\bullet(I)/K_\bullet(I)_{\ge 1} \cong R/I,
\end{equation}
and if $M$ is an $R$-module, we can apply $\Hom_{\Der(\ModR)}(-,\Sigma^i M)$ to obtain a map
\begin{equation}
\label{E:Koszul-compare}
q^i_M\colon \ext^i_R(R/I,M) \to H^i(I;M).
\end{equation}
This map is natural in $M$ and compares the $\ext$-group to the Koszul cohomology with coefficients in $M$ (see also~\cite[Proposition 1.6.9]{BH}). It is easy to see that $q^0_M$ is always an isomorphism, but for $i > 0$ the relation is more complicated and will be studied in the next subsection. To get a quick impression of the potential difficulties, note that $H^i(I;M)$ always commutes with direct limits, while $\ext^i_R(R/I,-)$ need not if $R$ is not a coherent ring.

\subsection{More on the relation between the Koszul cohomology and Ext groups}
\label{subsec:Koszul-Ext}
Our strategy will be to try to quantify the difference between $K_\bullet(I)$ and $R/I$ in the derived category. We start with an easy lemma.

	\begin{lem}
			\label{L:vanish}
		Let $I$ be a finitely generated ideal. Suppose that $M$ is an $R$-module such that $M \in \bigcap_{i=0}^{n-1}\Ker\ext_R^i(R/I,-)$. Then $\ext_R^k(N,M)=0$ for all $k=0,1,\ldots,n-1$ and any $R/I$-module $N$.
	\end{lem}
	\begin{proof}
		Since $N$ is an $R/I$-module, there is an exact sequence
		$$0 \rightarrow K \rightarrow R/I^{(\varkappa)} \rightarrow N \rightarrow 0.$$
		The lemma is then proved by applying $\Hom_R(-,M)$ on this sequence and a straightforward induction on $k$, using the fact that $K$ is also an $R/I$-module.
		\end{proof}

  The last result extends in a straightforward way to the derived category.
	
	\begin{lem}
			\label{L:vanish-derived}
		Let $I$ be a finitely generated ideal and suppose that $M$ is an $R$-module such that $M \in \bigcap_{i=0}^{n-1}\Ker\ext_R^i(R/I,-)$. If $X \in \Der(\ModR)$ is a complex such that $H_k(X) = 0$ for $k \le -n$ and $I \subseteq \Ann\big(H_k(X)\big)$ for $k = -n+1,\ldots,-1,0$. Then $\Hom_{\Der(\ModR)}(X,M) = 0$. 
	\end{lem}
	
	\begin{proof}
		Using the notation from \S\ref{subsec:derived} for truncations of complexes, we first recall the well-known fact that $\Hom_{\Der(\ModR)}(X_{\ge1},M)=0$; see~\cite[Propositions 3.1.8 and 3.1.10]{KS}. Since by our assumption $\Hom_{\Der(\ModR)}(\Sigma^k H_k(X),M) = 0$ for $k = -n+1,\ldots,-1,0$, a straightforward induction using the triangles from~\eqref{E:1-step-filt} in~\S\ref{subsec:derived} shows that $\Hom_{\Der(\ModR)}(X_{\ge-n+1},M) = 0$. Finally, the inclusion $X_{\ge-n+1} \to X$ is an isomorphism in $\Der(\ModR)$ since we assume that $H_k(X) = 0$ for $k \le -n$.
	\end{proof}

The latter lemma implies that the comparison map in~\eqref{E:Koszul-compare} in \S\ref{subsec:Koszul} between the $i$-th Koszul cohomology of $I$ and $\ext^i_R(R/I,M)$ is an isomorphism under certain assumptions.

\begin{prop} \label{P:Koszul-ext}
Let $R$ be a commutative ring and $I$ be a finitely generated ideal of $R$. Suppose that $M$ is a module such that $M \in \bigcap_{i=0}^{n-1} \Ker\ext_R^i(R/I,-)$ for some $n \geq 0$. Then $q^n_M\colon \ext_R^n(R/I,M) \to H^n(I;M)$ is an isomorphism.
\end{prop}

\begin{proof}
Since the two rightmost terms $F_1 \xrightarrow{d_1} F_0$ of $K_\bullet(I)$ constitute a projective presentation of $R/I$, it is clear that $q^0_M\colon H^0(I;M) \to \Hom_R(R/I,M)$ is always an isomorphism. This proves the lemma for $n=0$.

If $n>0$, we consider the triangle
$$
K_\bullet(I)_{\ge 1} \to K_\bullet(I) \xrightarrow{q} R/I \to \Sigma K_\bullet(I)_{\ge 1}
$$
in $\Der(\ModR)$ induced by the short exact sequence of complexes in the first three terms (see also~\eqref{E:Koszul-derived} in \S\ref{subsec:Koszul}). If we apply $\Hom_{\Der(\ModR)}(-,\Sigma^n M)$, we obtain an exact sequence of abelian groups
$$
\Hom(\Sigma K_\bullet(I)_{\ge 1},\Sigma^n M) \to \ext^n(R/I,M) \xrightarrow{q^n_M} H^n(I;M) \to \Hom(K_\bullet(I)_{\ge 1},\Sigma^n M).
$$
Since all the homologies of $K_\bullet(I)_{\ge 1}$ are $R/I$-modules, the leftmost and the rightmost term vanish by Lemma~\ref{L:vanish-derived} provided that $M \in \bigcap_{i=0}^{n-1} \Ker\ext_R^i(R/I,-)$. Hence $q^n_M$ is an isomorphism in such a case.
	\end{proof}

An immediate but particularly useful consequence is the following identification.

\begin{cor} \label{C:Koszul-vanish}
If $R$ is a commutative ring, then
$$
\bigcap_{i=0}^{n-1}\Ker\ext_R^i(R/I,-) = \bigcap_{i=0}^{n-1}\Ker H^i(I;-)
$$
as subcategories of $\ModR$ for each finitely generated ideal $I$ and $n \ge 0$.
\end{cor}
\begin{rem}
\label{R:Koszul-tor}
The dual versions of results from \S\ref{subsec:Koszul-Ext} also hold, and will be used in \S\ref{sec:derived}. Explicitly, for each finitely generated ideal $I$ and $n \geq 0$ we have 
$$
\bigcap_{i=0}^{n-1}\Ker\tor_i^R(R/I,-) = \bigcap_{i=0}^{n-1}\Ker H_i(I;-),
$$
and for any module $M$ belonging to this class, there is a natural isomorphism 
$$q^M_n: H_n(I;M) \rightarrow \tor_n^R(R/I,M),$$ 
obtained by applying $H_n(- \otimes_R^\mathbb{L} M)$ onto (\ref{E:Koszul-derived}). This can be proven either directly using similar arguments as in this section, or it follows by using the elementary duality $(-)^+$ (see \S\ref{sec:tilt} for the definition). Indeed, using the $\Hom$-$\otimes$-adjunction and exactness of the elementary duality, we have for any $M \in \bigcap_{i=0}^{n-1}\Ker\tor_i^R(R/I,-)$ that $M^+ \in \bigcap_{i=0}^{n-1}\Ker\ext_R^i(R/I,-)$. It is straightforward to check that the same properties ensure that $(q^M_n)^+$ is naturally equivalent to $q_{M^+}^n$, which is an isomorphism by Proposition~\ref{P:Koszul-ext}. Since $(-)^+$ is exact and faithful, we conclude that $q^M_n$ is an isomorphism.
	\end{rem}
\subsection{Vaguely associated primes revisited}

Now we give another way to express the very same class as in Corollary~\ref{C:Koszul-vanish} by giving a homological generalization of Proposition~\ref{P:torsion-free-Thomason}. We start with an easy observation

\begin{lem} \label{L:basic-hered-torsion}
Let $R$ be a commutative ring, $I$ be a finitely generated ideal. Then $\mathcal{F} = \Ker\Hom_R(R/I,-)$ is a hereditary torsion-free class of finite type whose corresponding Thomason set (in the sense of Proposition~\ref{P:torsion-Thomason}) is $V(I)$.
\end{lem}

\begin{proof}
By Lemma~\ref{L:Gabriel-finite}, the smallest Gabriel topology $\mathcal{G}$ containing $I$ consists of the ideals $J$ such that $I^m \subseteq J$ for some $m \ge 1$. The corresponding cyclic module $R/J$ then admits a filtration by $R/I$-modules of length $m$. Hence the smallest torsion class $\mathcal{T}$ containing $R/I$ coincides with the smallest torsion class containing $\mathcal{G}$. In particular, $(\mathcal{T},\mathcal{F})$ is a hereditary torsion pair. The fact that it corresponds to the Thomason set $V(I)$ follows from the proof of Proposition~\ref{P:torsion-Thomason}.
\end{proof}

Now can we state and prove the promised result.

\begin{prop} \label{P:grade-vass}
Let $R$ be a commutative ring, $I$ be a finitely generated ideal, $M$ be an $R$-module and $n \ge 1$. Then the following are equivalent:
\begin{enumerate}
\item $\ext_R^i(R/I,M) = 0$ for each $i = 0,1,\dots,n-1$.
\item $\vass\big(\Omega^{-i}(M)\big) \cap V(I) = \emptyset$ for each $i = 0,1,\dots,n-1$.
\end{enumerate}
\end{prop}

\begin{proof}
We prove the proposition by induction on $n$. Let $\mathcal{F} = \Ker\Hom_R(R/I,-)$. The statement for $n=1$ is precisely Proposition~\ref{P:torsion-free-Thomason}. Suppose now that $n>1$, and consider the exact sequence
\begin{equation} \label{E:grade-shift}
0 \to \Omega^{-(n-2)}(M) \to E \to \Omega^{-(n-1)}(M) \to 0,
\end{equation}
where $E$ is the injective envelope of $\Omega^{-(n-2)}(M)$. An application of $\Hom_R(R/I,-)$ yields an exact sequence
$$ \Hom_R(R/I,E) \to \Hom_R(R/I,\Omega^{-(n-1)}(M)) \to \ext_R^{n-1}(R/I,M) \to 0. $$

If (1) holds, we have $\vass\big(\Omega^{-(n-2)}(M)\big) \cap V(I) = \emptyset$ by the inductive hypothesis, so $\Omega^{-(n-2)}(M) \in \mathcal{F}$ by Proposition~\ref{P:torsion-free-Thomason}. Since $\mathcal{F}$ is a hereditary torsion-free class, also $E\in\mathcal{F}$ and the leftmost term in~\eqref{E:grade-shift} vanishes. As also $\ext_R^{n-1}(R/I,M) = 0$ by the assumption, we have $\Hom_R(R/I,\Omega^{-(n-1)}(M)) = 0$ and $\vass\big(\Omega^{-(n-1)}(M)\big) \cap V(I) = \emptyset$ as required.

Suppose conversely that (2) holds. Then $\Hom_R(R/I,\Omega^{-(n-1)}(M)) = 0$ by Proposition~\ref{P:torsion-free-Thomason} and hence $\ext_R^{n-1}(R/I,M) = 0$. The other $\ext$-groups in (1) vanish by the inductive hypothesis.
\end{proof}

\subsection{Characterizations of grade}

Now we are in a position to state and prove the main result of the section. As the definition of the grade of a module is specific to the noetherian situation and will not be so important for the rest of the paper, we only refer to~\cite[\S\S I.1.1--I.1.2]{BH} for the corresponding standard definitions. The equivalences below are well-known under additional finiteness conditions ($R$ noetherian, $M$ finitely generated)---see for instance~\cite[Theorems I.2.5 and I.6.17]{BH}.

\begin{thm} \label{T:grade}
Let $R$ be a commutative ring, $M$ be an $R$-module and $n \ge 1$. Then the following are equivalent:
\begin{enumerate}
\item $H^i(I;M) = 0$ for each $i = 0,1,\dots,n-1$.
\item $\ext_R^i(R/I,M) = 0$ for each $i = 0,1,\dots,n-1$.
\item $\vass\big(\Omega^{-i}(M)\big) \cap V(I) = \emptyset$ for each $i = 0,1,\dots,n-1$.
\end{enumerate}
If, moreover, $R$ is noetherian and $M$ finitely generated, the statements are further equivalent to:
\begin{enumerate}
\item[(4)] The grade of $I$ on $M$ is at least $n$.
\end{enumerate}
\end{thm}

\begin{proof}
The equivalence between (1) and (2) has been established in Corollary~\ref{C:Koszul-vanish} and the equivalence between (2) and (3) in Proposition~\ref{P:grade-vass}. For the equivalence between (2) and (4) see~\cite[Theorems I.2.5]{BH}.
\end{proof}

\section{Infinitely generated tilting theory}
\label{sec:tilt}
At this point, we quickly recall basic terminology and facts about module approximations and cotorsion pairs, two essential tools of the forthcoming sections. We also remind the reader of the notion of (not necessarily finitely generated) $n$-tilting and $n$-cotilting module, as defined by \cite{CT} and \cite{AC}, and the duality between those two.
\subsection{Module approximations} We briefly recall the definitions of (pre)covers and (pre)envelopes of modules. Let $\mathcal{C}$ be a class of right $R$-modules, and $M \in \ModR$. We say that a map $f: C \rightarrow M$ is a $\mathcal{C}$\emph{-precover} of $M$ provided that $C \in \mathcal{C}$, and for any $C' \in \mathcal{C}$ the map $\Hom_R(C',f)$ is surjective. Furthermore, if any map $g \in \End_R(C)$ such that $f=fg$ is necessarily an automorphism, we say that $f$ is a $\mathcal{C}$\emph{-cover}. Finally, a surjective map $f: C \rightarrow M$ is called a \emph{special $\mathcal{C}$-precover} if $C \in \mathcal{C}$ and $\Ker(f) \in {}^{\perp_1}\mathcal{C}$. It is easy to see that any special $\mathcal{C}$-precover is a $\mathcal{C}$-precover. Also, by the Wakamatsu Lemma (\cite[Lemma 5.13]{GT}), any surjective $\mathcal{C}$-cover is a special $\mathcal{C}$-precover. Finally, we say that a class $\mathcal{C}$ is \emph{(special) (pre)covering}, if any module $M \in \ModR$ admits a (special) $\mathcal{C}$-(pre)cover.

The notions of $\mathcal{C}$-preenvelope, $\mathcal{C}$-envelope, and special $\mathcal{C}$-preenvelope are defined dually.
\subsection{Cotorsion pairs}
Given a class $\mathcal{C} \subseteq \ModR$, we fix the following notation:
$$\mathcal{C}^{\perp_1} = \{M \in \ModR \mid \ext_R^1(C,M)=0 \text{ for all $C \in \mathcal{C}$}\},$$
$$\mathcal{C}^{\perp} = \{M \in \ModR \mid \ext_R^i(C,M)=0 \text{ for all $C \in \mathcal{C}$ and $i>0$}\},$$
$$\mathcal{C}^{\intercal} = \{M \in \RMod \mid \tor_i^R(C,M)=0 \text{ for all $C \in \mathcal{C}$ and $i>0$}\}.$$
We also define the ``left-hand'' version of those classes in an obviously analogous way. A pair of classes $(\mathcal{A},\mathcal{B})$ is called a \emph{cotorsion pair} if $\mathcal{B}=\mathcal{A}^{\perp_1}$, and $\mathcal{A}={}^{\perp_1}\mathcal{B}$. Such a cotorsion pair is said to be \emph{hereditary}, if furthermore $\mathcal{B}=\mathcal{A}^\perp$. 

A cotorsion pair $(\mathcal{A},\mathcal{B})$ is \emph{complete} provided that $\mathcal{A}$ is a special precovering class (equivalently, $\mathcal{B}$ is a special preenveloping class, see \cite[Lemma 5.20]{GT}). Complete cotorsion pairs are abundant - indeed, any cotorsion pair generated by a set is complete, and the left class of the pair can be described explicitly:

\begin{thm}
	\label{cotorsionpairs}
	\emph{(\cite[Theorem 6.11]{GT}, \cite[Corollary 6.14]{GT})} Let $\mathcal{S}$ be a set of modules. Then:
	\begin{enumerate}
		\item[(i)] The cotorsion pair $({}^{\perp_1}(S^{\perp_1}),S^{\perp_1})$ is complete,
		\item[(ii)] The class ${}^{\perp_1}(S^{\perp_1})$ consists precisely of all direct summands of all $\mathcal{S} \cup \{R\}$-filtered modules.
	\end{enumerate}
\end{thm}
\subsection{Tilting and cotilting modules and classes}
Let $T$ be a right $R$-module and $n \geq 0$. We say that $T$ is \emph{$n$-tilting} if the following three conditions hold:
\begin{enumerate}
	\item[(T1)] $\pd T \leq n$,
	\item[(T2)] $\ext_R^i(T,T^{(X)})=0$ for all $i>0$ and all sets $X$,
	\item[(T3)] there is an exact sequence $0 \rightarrow R \rightarrow T_0 \rightarrow \cdots \rightarrow T_n \rightarrow 0$, where $T_i$ is a direct summand of a direct sum of copies of $T$ for all $i=0,1,\ldots,n$.
\end{enumerate}
A module $T$ is \emph{tilting} if it is $n$-tilting for some $n \geq 0$. An $n$-tilting module $T$ induces a hereditary and complete cotorsion pair $(\mathcal{A},\mathcal{T})=({}^{\perp_1}(T^\perp), T^\perp)$. The class $\mathcal{T}$ is called an \emph{($n$-)tilting class}. Two tilting modules $T,T'$ are \emph{equivalent} if they induce the same tilting class. Even though the tilting modules in our setting are in general big (indeed, over a commutative ring, any finitely generated tilting module is projective), the tilting classes arise from small modules in the following sense. Let $\modR$ denote the full subcategory of $\ModR$ consisting of \emph{strongly finitely presented modules}, that is, modules having finite projective resolution consisting of finitely generated projectives. A full subcategory $\mathcal{S}$ of $\modR$ is called \emph{resolving} if it contains all finitely generated projectives, is closed under extensions and direct summands, and $A \in \mathcal{S}$ whenever there is an exact sequence
$$0 \rightarrow A \rightarrow B \rightarrow C \rightarrow 0,$$
with $B,C \in \mathcal{S}$.
\begin{thm}
		\label{T:finitetype}
		\emph{(\cite{BS}, \cite[Theorem 13.49]{GT})}
		There is a bijective correspondence between $n$-tilting classes $\mathcal{T}$ and resolving subcategories $\mathcal{S}$ of $\modR$ consisting of $R$-modules of projective dimension $\leq n$. The correspondence is given by mutually inverse assignments $\mathcal{T} \mapsto ({}^\perp \mathcal{T}) \cap \modR$, and $\mathcal{S} \mapsto \mathcal{S}^\perp = \mathcal{S}^{\perp_1}$.
\end{thm}

The cotilting modules have a formally dual definition - a left $R$-module $C$ is \emph{($n$-)cotilting} if the following conditions hold:
\begin{enumerate}
	\item[(C1)] $\id T \leq n$,
	\item[(C2)] $\ext_R^i(C^X,C)=0$ for all $i>0$ and all sets $X$,
	\item[(C3)] there is an exact sequence $0 \rightarrow C_n \rightarrow  \cdots \rightarrow C_0 \rightarrow W \rightarrow 0$, where $C_i$ is a direct summand of a direct product of copies of $T$ for all $i=0,1,\ldots,n$, and $W$ is an injective cogenerator of $\RMod$.
\end{enumerate}
As for the tilting modules, a cotilting module $C$ induces a \emph{cotilting class} $\mathcal{C}={}^\perp C$, and two cotilting modules $C,C'$ are \emph{equivalent} if they induce the same cotilting class. There is also an explicit duality between tilting and cotilting modules. If $R$ is a $k$-algebra over a commutative ring $k$ (e.g. $k=\mathbb{Z}$), we denote by $(-)^+=\Hom_k(-,E)$ the duality with respect to an injective cogenerator $E$ of $k\operatorname{-Mod}$. Then for any right tilting $R$-module $T$, the dual $T^+$ is a left cotilting $R$-module. We say that a tilting class $\mathcal{T}$ in $\ModR$ and cotilting class $\mathcal{C}$ in $\RMod$ are \emph{associated} if there is a tilting module $T$ inducing $\mathcal{T}$ such that $T^+$ induces $\mathcal{C}$.

It is not true that every cotilting class is associated to a tilting class. We say that a class $\mathcal{C}$ is of \emph{cofinite type} provided that there is a set of strongly finitely presented modules $\mathcal{S}$ of projective dimension bounded by $n$ such that $\mathcal{C} = \mathcal{S}^\intercal$.

\begin{thm}
		\label{T:cofinitetype}
		\emph{(\cite[Proposition 15.17]{GT}, \cite[Theorem 15.18]{GT})} Any class of cofinite type is cotilting. A cotilting class is associated to some tilting class if and only it is of cofinite type. Furthermore, the assignment $T^\perp \mapsto {}^\perp (T^+)$ induces a bijection between tilting class in $\ModR$ and cotilting classes in $\RMod$ of cofinite type.
\end{thm}

An example of a cotilting class not of cofinite type was first exhibited in \cite{B2}. In \S\ref{sec:notcofinite}, we show a more general construction of such classes.

All tilting and all cotilting classes are definable (i.e., closed under pure submodules, direct products, and direct limits). Furthermore, a pair of associated tilting class and cotilting class is dual definable in the following sense:
\begin{lem}
		\label{L78}
		Let $R$ be a ring, $\mathcal{T}$ a tilting class in $\ModR$, and $\mathcal{C}$ the cotilting class of cofinite type in $\RMod$ associated to $\mathcal{T}$. Then for any $M \in \ModR$, and $N \in \RMod$:
	\begin{enumerate}
			\item[(i)] $M \in \mathcal{T}$ if and only if $M^+ \in \mathcal{C}$,
			\item[(ii)] $N \in \mathcal{C}$ if and only if $N^+ \in \mathcal{T}$.
	\end{enumerate}
\end{lem}
\begin{proof}
		The proof is the same as that of \cite[Lemma 3.3]{AH}. As $\mathcal{T}$ is of finite type, there is a set $\mathcal{S} \subseteq \modR$ such that $\mathcal{T} = \mathcal{S}^\perp$. By \cite[Theorem 15.19, and the paragraph following it]{GT}, also $\mathcal{C}=\mathcal{S}^\intercal$. Using the Hom-$\otimes$ adjunction, exactness of $(-)^+$, and \cite[Theorem 3.2.10]{EJ}, we have $\tor_i^R(S,N)^+ \simeq \ext_R^i(S,N^+)$, which yields $(ii)$. Since $\mathcal{T}$ is definable, $M \in \mathcal{T}$ if and only if $M^{++} \in \mathcal{T}$ by \cite[3.4.21]{MP}. Then we can apply $(ii)$ tho obtain $(i)$.
\end{proof}
	Finally, as in \cite{APST}, we fix the following notation for cotilting classes of lower dimensions induced by a cotilting class, which will be useful for arguing by induction on the dimension.
\begin{notation}
		Given an $n$-cotilting class $\mathcal{C}$ induced by a cotilting module $C$, we let $\mathcal{C}_{(i)}={}^\perp \Omega^{-i}C$ for all $i\geq 0$. In particular, $\mathcal{C}_{(0)}=\mathcal{C}$, and $\mathcal{C}_{(i)}$ is a $(n-i)$-cotilting class for all $i=0,1,\ldots,n$ (see \cite[Lemma 3.5]{APST}).
\end{notation}

\section{Cotilting classes of cofinite type}
\label{sec:cotilt}
In this section we classify cotilting classes of cofinite type over a commutative ring. The parametrizing set for this classification will consist of sequences of torsion-free classes of hereditary torsion-free pairs of finite type satisfying some extra conditions. Using results from Section~\ref{sec:torsion-comm}, these classes are in bijective correspondence with certain Thomason sets, and thus generalize in a direct way the parametrizing sets used in the noetherian case in \cite{APST}.
\begin{definition}
	Let $R$ be a commutative ring and $n \geq 0$. We say that
			\item a sequence of torsion-free classes $\mathfrak{S}=(\mathcal{F}_0, \mathcal{F}_1, \ldots, \mathcal{F}_{n-1})$ is a \emph{characteristic sequence (of length n)} if 
					\begin{enumerate}
						\item[(i)] $\mathcal{F}_i$ is hereditary and of finite type for each $i=0,1,\ldots,n-1$,
						\item[(ii)] $\mathcal{F}_0 \subseteq \mathcal{F}_1 \subseteq \cdots \subseteq \mathcal{F}_{n-1}$,
						\item[(iii)] $\Omega^{-i} R \in \mathcal{F}_i$ for each $i=0,1,\ldots,n-1$.
					\end{enumerate}
\end{definition}
\begin{notation} \label{C(S)}
Given a characteristic sequence $\mathfrak{S}=(\mathcal{F}_0,\mathcal{F}_1,\ldots,\mathcal{F}_{n-1})$ and $i=0,1,\ldots,n-1$ we put
			$$\mathcal{F}_i(\mathfrak{S})=\mathcal{F}_i,$$
			and we define a class
			$$\mathcal{C}(\mathfrak{S})=\{M \in \ModR \mid \Omega^{-i}M \in \mathcal{F}_i(\mathfrak{S}) \text{ for each $i=0,1,\ldots,n-1$}\}.$$
\end{notation}
Our goal in this section is to prove the following:
\begin{thm}
	\label{T01}
	Let $R$ be a commutative ring and $n \geq 0$. Then the assignment 
		$$\Psi: \mathfrak{S} \mapsto \mathcal{C}(\mathfrak{S})$$
		is a bijection between the set of all characteristic sequences of length $n$ and all $n$-cotilting classes in $\ModR$ of cofinite type.
\end{thm}
The proof of Theorem~\ref{T01} will be done in several steps throughout this section, by proving subsequently that $\Psi$ is surjective, well-defined, and injective.
\subsection{$\Psi$ is surjective}
\begin{lem}
	\label{L00}
	Let $R$ be a commutative ring, $M \in \ModR$, and $P$ a finitely generated projective $R$-module. If $\iota: M \rightarrow E(M)$ is the injective envelope of $M$, then $P \otimes_R \iota$ is the injective envelope of $P \otimes_R M$.
\end{lem}
\begin{proof}
		Let $P' \in \ModR$ be a finitely generated projective module such that $P \oplus P' \simeq R^n$ for some $n \in \omega$. As $E(M)^n$ is injective, the map $R^n \otimes_R \iota = \iota^n$ is the injective envelope of $M^n$ by~\cite[Proposition 6.16(2)]{AF}. As $\iota^n = (P \otimes_R \iota) \oplus (P' \otimes_R \iota)$, we infer that $P \otimes_R \iota$ is the injective envelope of $P \otimes_R M$.
\end{proof}
\begin{prop}
		\label{P01}
		Let $R$ be a commutative ring and $\mathcal{C}$ be a cotilting class in $\ModR$. If $\mathcal{C}$ is of cofinite type, then it is closed under injective envelopes.
\end{prop}
\begin{proof}
		Fix a module $M \in \mathcal{C}$, and let us show that $E(M) \in \mathcal{C}$, provided that $\mathcal{C}$ is of cofinite type. Under this assumption, there is a set $\mathcal{S}$ of strongly finitely presented modules of projective dimension bounded by $n$ such that $\mathcal{C}=\mathcal{S}^\intercal$. We proceed by induction on $n$. If $n=0$ the claim is clear as $\mathcal{C} = \ModR$; henceforth assume that we proved the claim for all $k<n$. Pick $S \in \mathcal{S}$ and fix its projective resolution
		$$P_\bullet: \quad 0 \rightarrow P_n \xrightarrow{\sigma_n} P_{n-1} \xrightarrow{\sigma_{n-1}} \cdots \xrightarrow{\sigma_1} P_0 \rightarrow 0,$$
		consisting of finitely generated projectives. Tensoring the complex $P_\bullet$ with the injective envelope $\iota: M \rightarrow E(M)$ yields a commutative diagram in $\ModR$ (this is where we use the commutativity of $R$):
	$$\begin{CD}
		0       & &    0    & &    0 \\
		@VVV @VVV @VVV \\
		P_2 \otimes_R M @>\sigma_2 \otimes_R M>> P_1 \otimes_R M @>\sigma_1 \otimes_R M>> P_0 \otimes_R M \\
		@VVV @VP_1 \otimes_R\iota VV @VVV \\
		P_2 \otimes_R E(M) @>\sigma_2 \otimes_R E(M)>> P_1 \otimes_R E(M) @>\sigma_1 \otimes_R E(M)>> P_0 \otimes_R E(M) \\
\end{CD}$$
(In the case of $n=1$ we put $P_2=0$.)

Since $P_\bullet$ consists of projective modules, the columns are exact. We have 
		$$H_1(P_\bullet \otimes_R M) \simeq \tor_1^R(S,M) = 0,$$ 
		as $M \in \mathcal{C}$, and we want to show that $H_1(P_\bullet \otimes_R E(M)) \simeq \tor_1^R(S,E(M) = 0$. If $\Ker (\sigma_1 \otimes_R E(M))=0$ there is nothing to prove. Otherwise, since $P_1 \otimes_R \iota$ is the injective envelope of $P_1 \otimes_R M$ by Lemma~\ref{L00}, $P_1 \otimes_R \iota$ is an essential monomorphism. Therefore, the module
$$ \Ker (\sigma_1 \otimes_R E(M)) \cap \operatorname{Im}(P_1 \otimes_R \iota) = (P_1 \otimes_R \iota)(\Ker (\sigma_1 \otimes_R M))=(P_1 \otimes_R \iota)(\operatorname{Im} (\sigma_2 \otimes_R M)) $$
is non-zero, and thus essential in $\Ker (\sigma_1 \otimes_R E(M))$. It follows that $\operatorname{Im} (\sigma_2 \otimes_R E(M))$ is essential in $\Ker (\sigma_1 \otimes_R E(M))$. Now we use the induction hypothesis, which implies that $\tor_k^R(S,E(M)) \simeq \tor_1^R(\Omega^{k-1}S,E(M))=0$ for all $k>1$. This means that $H_k(P_\bullet \otimes_R E(M))=0$ for all $k>1$, and as this complex consists of injectives and is left-bounded, the map $(\sigma_2 \otimes_R E(M))$ is a split monomorphism. The only case for which this is not a contradiction is when $\operatorname{Im}(\sigma_2 \otimes_R E(M))=\Ker (\sigma_1 \otimes_R E(M))$, showing that $0=H_1(P_\bullet \otimes_R E(M))=\tor_1^R(S,E(M))$, proving finally that $E(M) \in \mathcal{C}$.
\end{proof}
Given a fixed module $C$, we can assign to any module $X$ the canonical map $\eta_X: X \rightarrow C^{\Hom_R(X,C)} = C_X$. This map is in fact (covariantly) functorial, as we recall in the following Lemma.
\begin{lem}
	\label{functorialmap}
	The map $\eta_X$ is functorial in $X$. That is, given any map $X \xrightarrow{f} Y$, there is a map $\eta_f: C_X \rightarrow C_Y$ such that $\eta_Y f = \eta_f \eta_X$, and such that $\eta_g \eta_f = \eta_{gf}$ for any map $g: Y \rightarrow Z$.
\end{lem}
\begin{proof}
		For any $\beta \in \Hom_R(Y,C)$, let $\pi_\beta: C_Y \rightarrow C$ be the projection onto the $\beta$-th coordinate. We define $\eta_f$ by the following rule: For any ${\bf c}=(c_\alpha)_{\alpha \in \Hom_R(X,C)} \in C_X$, and any $\beta \in \Hom_R(Y,C)$, we let
		\begin{equation}
				\label{functorialmapE01}
		\pi_\beta \eta_f({\bf c}) = \begin{cases}c_\alpha, & \alpha=\beta f \\ 0, & \text{otherwise}.\end{cases}
		\end{equation}
		It is easy to see that $\eta_f$ is an $R$-module homomorphism. Also, for any $\beta \in \Hom_R(Y,C)$ we have
		$$\pi_\beta \eta_Y f = \beta f = \pi_\beta \eta_f \eta_X,$$
		proving that indeed $\eta_Y f = \eta_f \eta_X$. Finally, the equality $\eta_g\eta_f=\eta_{gf}$ can be checked by direct computation from (\ref{functorialmapE01}).
\end{proof}
\begin{lem}
		\label{L01}
		Let $R$ be a ring and let $\mathcal{C}$ be an cotilting class in $\RMod$ closed under injective envelopes. Then there is a hereditary faithful torsion-free class of finite type $\mathcal{F}$ such that 
		$$\mathcal{C}=\{M \in \ModR \mid M\in \mathcal{F} ~\&~ \Omega^{-1}M \in \mathcal{C}_{(1)}\}.$$
\end{lem}
\begin{proof}
		Let $C$ be the cotilting module associated to $\mathcal{C}$. Applying $\Hom_R(-,C)$ to the exact sequence $0 \rightarrow M \rightarrow E(M) \rightarrow \Omega^{-1}M \rightarrow 0$, and using that $\mathcal{C}$ is closed under injective envelopes by Proposition~\ref{P01}, we infer that $M \in \mathcal{C}$ if and only if $E(M) \in \mathcal{C}$ and $\Omega^{-1}M \in \mathcal{C}_{(1)}$.

	We let $\mathcal{F}$ be the closure of $\mathcal{C}$ under submodules. As $\mathcal{C}$ is closed under injective envelopes, injectives of $\mathcal{C}$ and $\mathcal{F}$ coincide. From this it is easy to infer that $\mathcal{C}=\{M \in \ModR \mid M\in \mathcal{F} ~\&~ \Omega^{-1}M \in \mathcal{C}_{(1)}\}$. We are left to show that $\mathcal{F}$ is a hereditary faithful torsion-free class of finite type. It is easy to check that $R \in \mathcal{F}$, and that $\mathcal{F}$ is closed under submodules, injective envelopes, extensions, and products. 
	
		It remains to show that $\mathcal{F}$ is closed under direct limits. Note that $\mathcal{F}=\cogen(C)$. Let $(X_i)_{i \in I}$ be a directed system of modules from $\mathcal{F}$. As $X_i$ is cogenerated by $C$, the canonical map $\eta_{X_i}: X_i \rightarrow C_{X_i}$ is monic. Using the functoriality proved in Lemma~\ref{functorialmap}, we actually have a directed system $(X_i \rightarrow C_{X_i})_{i \in I}$ of monic maps. Taking the direct limit yields a monic map $\varinjlim_{i \in I}X_i \rightarrow \varinjlim_{i \in I}C_{X_i}$. As $\mathcal{C}$ is definable, the latter direct limit is in $\mathcal{C}$, proving that $\varinjlim_{i \in I} X_i$ is indeed in $\mathcal{F}$.
\end{proof}
\begin{lem}
		\label{L02}
		Let $R$ be a commutative ring and $\mathcal{C}$ be an $n$-cotilting class in $\ModR$ of cofinite type. Then there is a characteristic sequence $(\mathcal{F}_0,\mathcal{F}_1,\ldots,\mathcal{F}_{n-1})$ such that $\mathcal{C}=\{M \in \ModR \mid \Omega^{-i}M \in \mathcal{F}_i \text{ for each $i=0,1,\ldots,n-1$}\}$. In particular, map $\Psi$ is surjective.
\end{lem}
\begin{proof}
		First observe that $\mathcal{C}_{(1)}$ is of cofinite type provided that $\mathcal{C}$ is. Indeed, if $\mathcal{C}=\mathcal{S}^\intercal$ for some resolving subcategory of $\modR$ consisting of modules of bounded projective dimension, we have $\mathcal{C}_{(1)}=\{M \in \ModR \mid \Omega M \in \mathcal{C}\}=\{\Omega S \mid S \in \mathcal{S}\}^\intercal$, demonstrating the cofinite type of $\mathcal{C}_{(1)}$. By Proposition~\ref{P01}, we can apply Lemma~\ref{L01} inductively $(n-1)$-times in order to obtain the desired sequence $(\mathcal{F}_0,\mathcal{F}_1,\ldots,\mathcal{F}_{n-1})$, where $\mathcal{F}_{n-1} = \mathcal{C}_{(n-1)}$. Using the same Lemma and Proposition, this sequence is indeed characteristic, and $\mathcal{C}=\mathcal{C}((\mathcal{F}_0,\mathcal{F}_1,\ldots,\mathcal{F}_{n-1}))$ as desired.
\end{proof}
\subsection{$\Psi$ is well-defined}
	\begin{definition}
			Let $\mathfrak{S}$ be a characteristic sequence of length $n$. We let $\mathcal{G}_i(\mathfrak{S})$ denote the Gabriel topology associated to the torsion-free class $\mathcal{F}_i(\mathfrak{S})$ in the sense of Proposition~\ref{P:torsion} for each $i=0,1,\ldots,n-1$.

		In particular, $\mathcal{G}_i(\mathfrak{S})$ is a Gabriel topology of finite type, and 
		$$\mathcal{F}_i(\mathfrak{S})=\bigcap_{I \in \mathcal{G}_i^f(\mathfrak{S})} \Ker \Hom{}_R(R/I,-).$$
	\end{definition}
	\begin{lem}
			\label{L27}
		Let $\mathfrak{S}$ be a characteristic sequence of length $n$. Then 
		$$\mathcal{C}(\mathfrak{S})=\bigcap_{i=0}^{n-1}\bigcap_{I \in \mathcal{G}_i^{f}(\mathfrak{S})} \Ker \ext{}_R^i(R/I,-)$$.
	\end{lem}
	\begin{proof}
			Let $\mathfrak{S}=(\mathcal{F}_0, \mathcal{F}_1,\ldots,\mathcal{F}_{n-1})$. We prove by induction on $0 < k \leq n$ that $\mathcal{C}((\mathcal{F}_0, \mathcal{F}_1, \ldots, \mathcal{F}_{k-1})) = \bigcap_{i=0}^{k-1} \bigcap_{I \in \mathcal{G}_i^f(\mathfrak{S})} \Ker \ext{}_R^i(R/I,-)$. If $k=1$, then indeed $C((\mathcal{F}_0)) = \bigcap_{I \in \mathcal{G}_0^f(\mathfrak{S})} \Ker \Hom_R(R/I,-)$. Suppose that the claim is valid up to $k-1$ for $0<k-1<n$, and let $I \in \mathcal{G}_{k}^f$. Let $M \in C((\mathcal{F}_0, \mathcal{F}_1, \ldots, \mathcal{F}_{k-1}))$. The long exact sequence obtained by applying $\Hom_R(R/I,-)$ on the exact sequence $0 \rightarrow M \rightarrow E(M) \rightarrow \Omega^{-1}M \rightarrow 0$ yields an isomorphism $\ext_R^{k-1}(R/I, \Omega^{-1}M) \simeq \ext_R^k(R/I,M)$ (for $k=1$ we use the fact that $E(M) \in \mathcal{F}_0 \subseteq \mathcal{F}_k$). 
			
			By dimension shifting, we have $\ext_R^{k-1}(R/I,\Omega^{-1}M) \simeq \ext_R^1(R/I, \Omega^{-k+1}M)$. Finally, since $E(\Omega^{-k+1} M) \in \mathcal{F}_{k-1} \subseteq \mathcal{F}_k$, we have also the zero-dimension shift isomorphism $\ext_R^1(R/I, \Omega^{-k+1}M) \simeq \Hom_R(R/I, \Omega^{-k}M)$. Putting the isomorphisms together, we have $\Hom_R(R/I, \Omega^{-k} M) \simeq \ext_R^k(R/I, M)$, showing that $\Omega^{-k} M \in \mathcal{F}_k$ if and only if $M \in \bigcap_{I \in \mathcal{G}_k^f} \Ker \ext_R^k(R/I,-)$ for any $M \in \mathcal{C}(\mathcal{F}_0,\mathcal{F}_1,\ldots, \mathcal{F}_{k-1})$. Therefore, using the induction premise, an arbitrary module $M \in \ModR$ satisfies $M \in \mathcal{C}(\mathcal{F}_0,\mathcal{F}_1,\ldots,\mathcal{F}_k)$ if and only if $M \in \bigcap_{i=0}^{k} \bigcap_{I \in \mathcal{G}_i^f(\mathfrak{S})} \Ker \ext{}_R^i(R/I,-)$, establishing the induction step.
	\end{proof}

	In what follows, we denote by $(-)^*$ the regular module duality $\Hom_R(-,R)$.
	\begin{definition}
			\label{D11}
		Let $I$ be a finitely generated ideal, and let us denote the Koszul complex $K_\bullet(I)$ as follows
		$$\cdots \xrightarrow {d_{n+1}} F_n \xrightarrow {d_n} \cdots \xrightarrow{d_2} F_1 \xrightarrow{d_1} F_0 \rightarrow 0,$$
		where $F_k$ is in degree $k$ for all $k=0,1,2,\ldots,n+1$. We denote by $S_{I,k}$ the cokernel of the map $d^*_{k}$. That is, $S_{I,k}$ is (up to stable equivalence) the Auslander-Bridger transpose of $\Coker(d_{k})$.
	\end{definition}
	\begin{prop}
		\label{P14}
		Let $I$ be a finitely generated ideal such that $\ext_R^i(R/I,R)=0$ for all $i=0,1,\ldots,n-1$. Then:
		\begin{enumerate}
				\item[(i)] $S_{I,n}$ is a strongly finitely presented module of projective dimension $n$,
				\item[(ii)] $\bigcap_{i=0}^{n-1}\Ker \ext_R^i(R/I,-) = (S_{I,n})^\intercal$,
				\item[(iii)] $\bigcap_{i=0}^{n-1}\Ker \ext_R^i(R/I,-)$ is an $n$-cotilting class of cofinite type.
		\end{enumerate}
		\begin{proof}
				Let us adopt the notation for $K_\bullet(I)$ as in Definition~\ref{D11}. Let $0 < k \leq n$. Applying $(-)^*$ to $K_\bullet(I)$ we obtain complex
				\begin{equation}\label{E11}
		0 \rightarrow F_0^* \rightarrow F_{1}^* \rightarrow \cdots \rightarrow F_{k-1}^* \xrightarrow{d^*_n} F_k^* \rightarrow S_{I,k} \rightarrow 0,
\end{equation}
				which is exact by our assumption and Proposition~\ref{P:Koszul-ext}. This proves $(i)$.

				Denote by $C_\bullet$ the complex (\ref{E11}) with $S_{I,k}$ deleted, and $F_k^*$ in the degree zero. Then $C_\bullet$ is a projective resolution of $S_{I,k}$, and thus $H_1(C_\bullet \otimes_R M) \simeq \tor_1^R(S_{I,k},M)$. But as $F_j$ is finitely generated projective for all $j=0,1,\ldots,n$, we have by \cite[Proposition 20.10]{AF} a natural isomorphism $F_j^* \otimes_R M \simeq \Hom_R(F_j,M)$. Hence, using Proposition~\ref{P:Koszul-ext} we obtain
		\begin{equation}\label{E12}{\tor}_1^R(S_{I,k},M) \simeq H_1(C_\bullet \otimes_R M) \simeq \end{equation}
				$$ \simeq H^{k-1}({\Hom}_R(K_\bullet(I),M) \simeq {\ext}_R^{k-1}(R/I,M),$$
				for any $M \in \bigcap_{i=0}^{k-2}\Ker\ext_R^i(R/I,-)$.

				Now we prove $(ii)$. Note first that $S_{I,k}$ is an $(n-k)$-th syzygy of $S_{I,n}$. Hence, $\tor_k^R(S_{I,n},M) \simeq \tor_1^R(S_{I,n-k+1},M)$. Then $(S_{I,n})^\intercal = \bigcap_{k=1}^n \Ker \tor_1^R(S_{I,k},-)$. A straightforward induction on $k=1,2,\ldots,n$ together with (\ref{E12}) proves that the latter class is equal to $\bigcap_{k=1}^n \Ker\ext_R^{k-1} (R/I,-)$ as desired.

				Finally, $(iii)$ follows directly from $(ii)$ by Theorem~\ref{T:cofinitetype}. 
		\end{proof}
	\end{prop}
	\begin{lem}
			\label{L28}
			Let $\mathfrak{S}$ be a characteristic sequence of length $n$, then $\mathcal{C}(\mathfrak{S})$ is a $n$-cotilting class of cofinite type. In particular, map $\Psi$ is well-defined.
	\end{lem}
	\begin{proof}
			We have the following chain of equalities:
			$$\mathcal{C}(\mathfrak{S})=\bigcap_{i=0}^{n-1}\bigcap_{I \in \mathcal{G}_i^{f}(\mathfrak{S})} \Ker \ext{}_R^i(R/I,-) = \bigcap_{i=0}^{n-1}\bigcap_{I \in \mathcal{G}_i^{f}(\mathfrak{S})} \bigcap_{j=0}^i \Ker \ext{}_R^j(R/I,-)=$$
	$$=\bigcap_{i=0}^{n-1}\bigcap_{I \in \mathcal{G}_i^{f}(\mathfrak{S})} (S_{I,i+1})^\intercal.$$
	The first equality is Lemma~\ref{L27}, the second one follows easily from $\mathcal{G}_i^f \supseteq \mathcal{G}_{i+1}^f$ for each $i=0,1,\ldots,n-2$, and the last one is an application of Proposition~\ref{P14}(ii). Then Theorem~\ref{T:cofinitetype} yields the result.
	\end{proof}
\subsection{$\Psi$ is injective}
\begin{lem}
	\label{L31}
	Let $\mathfrak{S}=(\mathcal{F}_0,\mathcal{F}_1,\ldots,\mathcal{F}_{n-1})$ be a characteristic sequence, and $\mathcal{C}=\mathcal{C}(\mathfrak{S})$ the associated $n$-cotilting class. Then 
	\begin{enumerate}
			\item[(i)] $\Omega M \in \mathcal{C}((\mathcal{F}_0,\mathcal{F}_1,\ldots,\mathcal{F}_{n-1}))$ if and only if $M \in \mathcal{C}((\mathcal{F}_1,\mathcal{F}_2,\ldots,\mathcal{F}_{n-1}))$,
		\item[(ii)] $\mathcal{C}_{(i)}=\mathcal{C}((\mathcal{F}_i,\mathcal{F}_{i+1},\ldots,\mathcal{F}_{n-1}))$ for each $i=0,1,\ldots,n$.
	\end{enumerate}
\end{lem}
\begin{proof}
	\begin{enumerate}
			\item[(i)] Choose $I \in \mathcal{G}_i^f(\mathfrak{S})$ for some $i=0,1,\ldots,n-1$. Consider the long exact sequence obtained by applying $\Hom_R(R/I,-)$ onto a projective presentation of $M$, say
					$$0 \rightarrow \Omega M \rightarrow P \rightarrow M \rightarrow 0.$$
					Since cotilting classes contain all projectives modules, this long exact sequence yields $\Hom_R(R/I,\Omega M)=0$, and $\ext_R^j(R/I, \Omega M) \simeq \ext_R^{j-1}(R/I, M)$ for all $j=1,2,\ldots,i$. Therefore, 
					$$\Omega M \in \bigcap_{i=0}^{n-1}\bigcap_{I \in \mathcal{G}_i^{f}(\mathfrak{S})} \Ker \ext{}_R^i(R/I,-)$$ 
					if and only if 
					$$M \in \bigcap_{i=1}^{n-1}\bigcap_{I \in \mathcal{G}_i^{f}(\mathfrak{S})} \Ker \ext{}_R^{i-1}(R/I,-).$$ 
					This concludes $(i)$ by Lemma~\ref{L27}.

			\item[(ii)] It follows directly from the definition that $\mathcal{C}_{(i)}=\{M \in \ModR \mid \Omega^i M \in \mathcal{C}\}$. Then $(ii)$ is proved by $(i)$ and a straightforward induction.
	\end{enumerate}
\end{proof}
\begin{lem}
		\label{L:psiinjective}
		Let $\mathfrak{S}$ and $\mathfrak{S}'$ be two characteristic sequences. If $\mathfrak{S} \neq \mathfrak{S}'$, then $\mathcal{C}(\mathfrak{S}) \neq \mathcal{C}(\mathfrak{S}')$. In particular, map $\Psi$ is injective.
\end{lem}
\begin{proof}
		Let $i \geq 0$ be smallest such that $\mathcal{F}_i(\mathfrak{S}) \neq \mathcal{F}_i(\mathfrak{S}')$. Suppose without loss of generality that there is $M \in \mathcal{F}'_i(\mathfrak{S}) \setminus \mathcal{F}_i(\mathfrak{S})$. If $i=0$, then clearly $E(M)$ is in $\mathcal{C}(\mathfrak{S}')$, but $E(M) \not\in \mathcal{F}_0(\mathfrak{S})$ proving the statement for $i=0$. Let now $i=0,1,\ldots,n-1$ and suppose towards contradiction that $\mathcal{C}(\mathfrak{S})=\mathcal{C}(\mathfrak{S}')$. Then also $\mathcal{C}(\mathfrak{S})_{(i)}=\mathcal{C}(\mathfrak{S}')_{(i)}$, but this is a contradiction using the case $i=0$ and Lemma~\ref{L31}(ii).
\end{proof}
\subsection{The result}
Now we are ready to prove our classification theorem.
\begin{proof}[Proof of Theorem~\ref{T01}]
		The assignment $\Psi: \mathfrak{S} \mapsto \mathcal{S}(\mathfrak{S}))$ is a well-defined map from the set of all characteristic sequences of length $n$ to $n$-cotilting classes of cofinite type by Lemma~\ref{L28}. This map is injective by Lemma~\ref{L:psiinjective} and surjective by Lemma~\ref{L02}.
\end{proof}
\section{Main classification results}
\label{sec:main-results}
In this section we rephrase Theorem~\ref{T01} in terms of Thomason sets, and state our characterization of tilting classes over commutative rings.
\begin{thm}
	\label{T:mainthm1}
		Let $R$ be a commutative ring and $n\geq 0$. There is a 1-1 correspondence between $n$-cotilting classes $\mathcal{C}$ of cofinite type, and finite sequences $(X_0,X_1,\ldots,X_{n-1})$ of Thomason subsets of $\spec(R)$ satisfying:
		\begin{enumerate}
			\item[(i)] $X_0 \supseteq X_1 \supseteq \cdots \supseteq X_{n-1}$,
			\item[(ii)] $X_i \cap \vass(\Omega^{-j} R) = \emptyset$ for all $j=0,1,\ldots,i$ and all $i=0,1,\ldots,n-1$.
		\end{enumerate}
		The correspondence is given by mutually inverse maps
		$$\mathcal{C} \mapsto (\spec(R) \setminus \ass(\mathcal{C}_{(0)}),\spec(R) \setminus \ass(\mathcal{C}_{(1)}),\ldots,\spec(R) \setminus \ass(\mathcal{C}_{(n-1)})),$$
		$$(X_0,X_1,\ldots,X_{n-1}) \mapsto \{M \in \ModR \mid \vass(\Omega^{-i} M) \cap X_i = \emptyset \text{ for all $i=0,1,\ldots,n-1$}\}.$$
\end{thm}
\begin{proof}
		Start with a cofinite-type cotilting class $\mathcal{C}$, and let $\mathfrak{S}=(\mathcal{F}_0,\mathcal{F}_1,\ldots,\mathcal{F}_{n-1})$ be the characteristic sequence corresponding to $\mathcal{C}$ in the sense of Theorem~\ref{T01}. Note that $\ass(\mathcal{C}_{(j)}) = \ass(\mathcal{F}_j)$ for each $j=0,1,\ldots,n-1$. Indeed, one inclusion follows trivially, as $\mathcal{C}_{(j)} \subseteq \mathcal{F}_j$ by Lemma~\ref{L31}, while the second inclusion follows from $\mathcal{C}_{(j)}$ and $\mathcal{F}_j$ being both closed under injective envelopes and having the same injectives. 		
		
		The rest of the proof is a combination of Theorem~\ref{T01}, Proposition~\ref{P:torsion-Thomason}, and Proposition~\ref{P:torsion-free-Thomason}.
\end{proof}
\begin{thm}
		\label{T:mainthm2}
	Let $R$ be a commutative ring and $n \geq 0$. There are 1-1 correspondences between the following collections:
	\begin{enumerate}
			\item[(i)] sequences $(\mathcal{G}_0,\mathcal{G}_1,\ldots,\mathcal{G}_{n-1})$ of Gabriel topologies of finite type satisfying:
					\begin{enumerate}
							\item $\mathcal{G}_0 \supseteq \mathcal{G}_1 \supseteq \cdots \supseteq \mathcal{G}_{n-1}$,
							\item $\ext_R^j(R/I,R)=0$ for all $I \in \mathcal{G}_i$, all $i=0,1,\ldots,n-1$, and all $j=0,1,\ldots,i$.
					\end{enumerate}
			\item[(ii)] $n$-cotilting classes $\mathcal{C}$ in $\ModR$ of cofinite type,
			\item[(iii)] $n$-tilting classes $\mathcal{T}$ in $\ModR$,
			\item[(iv)] resolving subcategories $\mathcal{S}$ of $\modR$ consisting of modules of projective dimension $\leq n$.
	\end{enumerate}
	The correspondences are given as follows:
	$$(i) \rightarrow (ii): (\mathcal{G}_0,\mathcal{G}_1,\ldots,\mathcal{G}_{n-1}) \mapsto \bigcap_{i=0}^{n-1} \bigcap_{I \in \mathcal{G}_i^f} \Ker \ext{}_R^i(R/I,-) = \bigcap_{i=0}^{n-1} \bigcap_{I \in \mathcal{G}_i^f} (S_{I,i+1})^\intercal$$
	$$(i) \rightarrow (iii): (\mathcal{G}_0,\mathcal{G}_1,\ldots,\mathcal{G}_{n-1}) \mapsto \bigcap_{i=0}^{n-1} \bigcap_{I \in \mathcal{G}_i^f} \Ker \tor{}_i^R(R/I,-) = \bigcap_{i=0}^{n-1} \bigcap_{I \in \mathcal{G}_i^f} (S_{I,i+1})^\perp$$
	$$(i) \rightarrow (iv): (\mathcal{G}_0,\mathcal{G}_1,\ldots,\mathcal{G}_{n-1}) \mapsto \{M \in \modR \mid \text{$M$ is isomorphic}$$ $$\text{to a direct summand of a finitely $\{R\} \cup \{S_{I,i+1} \mid I \in \mathcal{G}_{i}^f, i<n\}$-filtered module}\}$$
\end{thm}
\begin{proof}
		Correspondence $(i) \rightarrow (ii)$ follows directly from Theorem~\ref{T01}, Gabriel correspondence between hereditary torsion pairs and Gabriel topologies (see \cite[VI, Theorem 5.1 and XIII, Proposition 1.2]{St}), and combination of Lemma~\ref{L27} and Proposition~\ref{P14}. We show that $(i) \rightarrow (iii)$ is a composition of $(i) \rightarrow (ii)$ with the character duality correspondence between tilting and cofinite-type cotilting classes. Indeed, we have as in the proof of Lemma~\ref{L78} that $\tor_i^R(X,M)=0$ if and only if $\ext_R^i(X,M^+)=0$ for any $i \geq 0$. Therefore, $M \in \bigcap_{i=0}^{n-1} \bigcap_{I \in \mathcal{G}_i^f} \Ker\tor{}_i^R(R/I,-)$ if and only if $M^+ \in \bigcap_{i=0}^{n-1} \bigcap_{I \in \mathcal{G}_i^f} \Ker\ext{}_R^i(R/I,-)$, and thus the former class is the tilting class associated to the latter cotilting class by Lemma~\ref{L78}(i). Similarly, $N \in \bigcap_{i=0}^{n-1} \bigcap_{I \in \mathcal{G}_i^f} (S_{I,i+1})^\intercal$ if and only if $N^+ \in \bigcap_{i=0}^{n-1} \bigcap_{I \in \mathcal{G}_i^f} (S_{I,i+1})^\perp$. As the latter class is of finite type, and thus definable, it is uniquely determined by its pure-injective objects (\cite[Lemma 6.9]{GT}). Using Lemma~\ref{L78}(ii) and \cite[Theorem 2.27(c)]{GT}, we see that pure-injectives of this class coincide with pure-injectives of the tilting class associated to the cotilting class $\bigcap_{i=0}^{n-1} \bigcap_{I \in \mathcal{G}_i^f} (S_{I,i+1})^\intercal$, and thus the two classes coincide.
		
		Finally, by Theorem~\ref{T:finitetype}, resolving subcategories $\mathcal{S}$ as in $(iv)$ correspond bijectively to tilting classes $\mathcal{T}$ via the assignment $\mathcal{T} \mapsto ({}^\perp \mathcal{T}) \cap \modR$. Whenever $\mathcal{T}=\mathcal{S}^\perp$ for some set $\mathcal{S} \subseteq \modR$, we have that ${}^\perp \mathcal{T}$ equals to all direct summands of all $\{R\} \cup \mathcal{S}$-filtered modules by Theorem~\ref{cotorsionpairs}, and thus ${}^\perp \mathcal{T} \cap \modR$ equals to direct summands of all finitely $\{R\} \cup \mathcal{S}$-filtered modules. By $(i) \rightarrow (iii)$, we can chose $\mathcal{S}=\{S_{I,i+1} \mid I \in \mathcal{G}_i^f, i=0,1,\ldots,n-1\}$, establishing $(i) \rightarrow (iv)$.
\end{proof}
\begin{rem}
		If the ring $R$ is coherent, we can use a projective resolution of $R/I$ in $\modR$ instead of the Koszul complex. Therefore, in this case $S_{I,i+1}$ can be replaced by module $\tr(\Omega^i R/I)$ for each $i=0,1,\ldots,n-1$, where $\tr$ is the Auslander-Bridger transpose (cf. \cite{APST}).
\end{rem}
\begin{ex}
		\label{ex:perfect}
		Let $R$ be a commutative perfect ring. Then the only tilting class in $\ModR$ is the trivial one, that is the whole $\ModR$. By Theorem~\ref{T:mainthm2}, it is enough to show that if $I$ is a finitely generated ideal such that $\Hom_R(R/I,R)=0$, then $I=R$. Indeed, since $I$ is finitely generated, the descending chain $(I^n \mid n \in \omega)$ stabilizes at some $n$ (see \cite[Theorem 23.20, p. 345]{LAM}). Then either $I$ is nilpotent, and thus $\Hom_R(R/I,R) \neq 0$ unless $I=R$, or $I^n$ is idempotent. As $I$ is finitely generated, the latter case implies that $I^n$ is a direct summand of $R$, and thus again $\Hom_R(R/I,R) \neq 0$ unless $I=R$.
\end{ex}
\section{Derived functors of torsion and completion and \v{C}ech (co)homology}
\label{sec:derived}
	In Section~\ref{sec:main-results}, we characterized cofinite type $n$-cotilting classes over a commutative ring as classes of all modules which vanish in certain degrees of Koszul cohomologies, arising from a set of finitely generated ideals. In this section, we show that we can replace Koszul complexes by two kinds of more canonically defined cohomology theories associated to an ideal---\v{C}ech cohomology, and local cohomology. Our results are valid for a general commutative ring, even though the latter two cohomology theories do not in general coincide for non-noetherian rings. An analogous description of $n$-tilting classes via \v{C}ech homology and local homology will also be accomplished, after dealing with a few extra technical difficulties. The main sources we use in this section are \cite{GM,Pos,PSY,Schenz}.

	Throughout this section, let $R$ be a commutative ring.
\subsection{Local (co)homology}
\label{subsec:loccohom}
	Given a finitely generated ideal $I$, there are two additive functors $\ModR \rightarrow \ModR$ arising from it---the $I$-torsion functor $\Gamma_I$ and the $I$-adic completion functor $\Lambda_I$, defined for an $R$-module $M$ as follows:
	$$\Gamma_I(M) = \{m \in M \mid I^nm = 0 \text{ for some $n>0$}\} = \varinjlim_{n \in \omega} \Hom_R(R/I^n,M),$$
	$$\Lambda_I(M) = \varprojlim_{n \in \omega} M \otimes_R R/I^n.$$
	A module $M$ is said to be $I$-torsion if $\Gamma_I(M)=M$, and we denote the full subcategory of all $I$-torsion modules by $\mathcal{T}_I$. Then $\mathcal{T}_I$ is an abelian category with exact direct sums, and the embedding $\mathcal{T}_I \subseteq \ModR$ is exact and clearly admits $\Gamma_I$ as is its right adjoint.
	\begin{equation}
		\label{E:torsionadj}
	\begin{tikzcd}
			\mathcal{T}_I \arrow[bend left]{r}[above]{\subseteq} & \ModR \arrow[bend left]{l}[below]{\Gamma_I}
	\end{tikzcd}
	\end{equation}
	In particular, $\Gamma_I$ is left exact, and we can form the right derived functor $\mathbb{R}\Gamma_I$, called the \emph{local cohomology} functor of $I$. 

	The situation is a bit more tricky in the case of completion functors. Following Positselski \cite{Pos}, we say that a module $M$ is an $I$-contramodule provided that $\ext_R^j(R[x_i^{-1}],M)=0$ for $j=0,1$, and for $i=1,2,\ldots,n$, where $\{x_1,x_2,\ldots,x_n\}$ is a set of generators of $I$. By \cite[p.~3880]{Pos}, the choice of generators does not matter, and this is a correct definition. Denote by $\mathcal{C}_I$ the full subcategory of all $I$-contramodules. Then $\mathcal{C}_I$ is an abelian category with exact products and the embedding $\mathcal{C}_I \subseteq \ModR$ is exact and admits a left adjoint (\cite[Proposition 2.1]{Pos}). Following \cite{Pos} we denote the left adjoint by $\Delta_I$.
	\begin{equation}
		\label{E:contraadj}
	\begin{tikzcd}
			\mathcal{C}_I \arrow[bend right]{r}[below]{\subseteq} & \ModR \arrow[bend right]{l}[above]{\Delta_I}
	\end{tikzcd}
	\end{equation}
	However, usually $\Lambda_I$ does not fit in this adjunction in place of $\Delta_I$. Indeed, $\Lambda_I$ can be neither left nor right exact, even over a noetherian ring. Nevertheless, we can compute the left derived functor $\mathbb{L}\Lambda_I$, and call it the \emph{local homology}. 
	
	By \cite[\S I. Lemma 5.13]{LMS}, both adjunctions \ref{E:torsionadj} and \ref{E:contraadj} survive passing to the (unbounded) derived category, and thus we have adjoint pairs:
	\begin{equation}
		\label{E:der-torsionadj}
	\begin{tikzcd}
			\Der(\mathcal{T}_I) \arrow[bend left]{r}[above]{} & \Der(\ModR) \arrow[bend left]{l}[below]{\mathbb{R}\Gamma_I}
	\end{tikzcd}
	\end{equation}
	and
	\begin{equation}
		\label{E:der-contraadj}
	\begin{tikzcd}
			\Der(\mathcal{C}_I) \arrow[bend right]{r}[below]{} & \Der(\ModR) \arrow[bend right]{l}[above]{\mathbb{L}\Delta_I}.
	\end{tikzcd}
	\end{equation}

	However, note that in general not even in the derived picture we can swap $\mathbb{L}\Delta_I$ for $\mathbb{L}\Lambda_I$ (see \cite[Example 2.6]{Pos}). This will be further discussed in~\S\ref{subsec:picture}.

\subsection{\v{C}ech (co)homology}
	As discussed in \S\ref{subsec:Koszul}, the Koszul complex $K_\bullet(I)$ for a finitely generated ideal $I$ is not a well-defined object even in the derived category, as the homology can change when passing from one generating system of $I$ to another. This can be mended by stepping outside of the realm of perfect complexes and using generators of $I$ to form a \v{C}ech cochain complex (also called a \emph{stable Koszul complex}).
	\begin{definition}
			Let $x$ be an element of $R$. The \emph{\v{C}ech complex} with respect to $x$, denoted by $\check{C}^\bullet(x)$, which is defined as
			$$0 \rightarrow R \xrightarrow{\iota} R_x \rightarrow 0,$$
			where $R_x = R[x^{-1}]$, $\iota$ is the natural morphism, and the cochain complex is concentrated in (cohomological) degrees 0 and 1. Given a sequence $\mathbf{x}=(x_1,x_2,\ldots,x_n)$ of elements of $R$, we define $\check{C}^\bullet(\textbf{x})$ as the tensor product $\bigotimes_{i=1}^n \check{C}^\bullet(x_i)$.
	\end{definition}
	\begin{lem}\emph{(\cite[Corollary 3.12]{Gfbs})}\label{lem:quasi-iso-Cech}
			Let $I$ be a finitely generated ideal, and $\mathbf{x}=(x_1,x_2,\dots,x_n), \mathbf{y}=(y_1,y_2,\ldots,y_m)$ two sequences of generators of $I$. Then the \v{C}ech complexes $\check{C}^\bullet(\mathbf{x})$ and $\check{C}^\bullet(\mathbf{y})$ are quasi-isomorphic.
	\end{lem}
	Lemma~\ref{lem:quasi-iso-Cech} legitimizes the following notation: Given a finitely generated ideal $I$ with a finite sequence of generators $\mathbf{x}$, we denote $\check{C}^\bullet(I) = \check{C}^\bullet(\mathbf{x})$. The cochain complex $\check{C}^\bullet(I)$ is then well-defined as an object of the derived category.

	Similar to Koszul complexes, we can compute \v{C}ech cohomology and homology. First we address the former, defined as follows:
	$$\check{H}^i(I;-)=H^{i}(\check{C}^\bullet(I) \otimes_R^\mathbb{L} -).$$ 
	As $\check{C}^\bullet(I)$ is a bounded complex of flat modules, we can drop the left derivation symbol $\mathbb{L}$ from the formula. This \emph{\v{C}ech cohomology} can also be viewed as a limit version of Koszul cohomology in the following way. Given $x \in R$, the Koszul chain complexes $(K_\bullet(x^j) \mid j>0)$ form an inverse system with the following connecting maps:
	\begin{equation}
		\begin{tikzcd}
				0 \arrow{r} & R \arrow{r}{\cdot x^j} & R\arrow{r}  & 0 \\
				0 \arrow{r} & R\arrow{u}{\cdot x} \arrow{r}{\cdot x^{j+1}} & R \arrow[-,double equal sign distance]{u}\arrow{r}& 0 
		\end{tikzcd}
	\end{equation}
	Dualizing this with respect to $R$, we obtain a direct system $(K_\bullet(x^j)^* \mid j>0)$ of cochain complexes, and it is easy to check that its limit is precisely $\check{C}^\bullet(x)$. As direct limit commutes with tensor product, we get $\check{C}^\bullet(I)=\varinjlim_{j>0} K_\bullet(\mathbf{x}^j)^*$, where we fix the notation
	$$\mathbf{x}^j=(x_1^j,x_2^j,\ldots,x_m^j)$$ 
	for a set of generators $\mathbf{x}=\mathbf{x}^1$ of $I$. Given a module $M$, we infer from exactness of the direct limit functor the following isomorphism
	\begin{equation}
			\label{E:ltensor}
			\check{H}^i(I;M) \simeq \varinjlim_{j>0} H^i(\mathbf{x}^j;M).
	\end{equation}
	Using this representation, we can already show that the \v{C}ech complexes classify the cofinite type cotilting classes (see also Theorem~\ref{T:mainthm3} below). In the following proofs, let always $I_j$ denote the ideal generated by the sequence $\mathbf{x}^j$ (not to be confused with $I^j$).
	\begin{lem}
\label{lem:Cech-cohom}
		Let $R$ be a commutative ring, $I$ a finitely generated ideal, and $n > 0$. Then 
			$$\bigcap_{i=0}^{n-1} \Ker H^i(I;-) = \bigcap_{i=0}^{n-1} \Ker \check{H}^{i}(I;-).$$
	\end{lem}
	\begin{proof}
We proceed by induction on $n \geq 0$. For the induction step, we fix throughout the proof a module
$$M \in \bigcap_{i=0}^{n-1}\Ker H^i(I;-) = \bigcap_{i=0}^{n-1}\Ker \check{H}^{i}(I;-)$$ 
(which is a vacuous assumption if $n=0$), and prove that $H^n(I;M)=0$ if and only if $\check{H}^n(I;M)=0$. We recall the comparison map $q^n_M(j): \ext_R^n(R/I_j,M) \rightarrow H^n(\mathbf{x}^j;M)$ from \S\ref{subsec:Koszul}. By Proposition~\ref{P:Koszul-ext}, the map $q^n_M(j)$ is an isomorphism for any $j>0$.

	Assume first that $H^n(I;M)=0$. By the pigeon hole principle, $I_j/I_{j+1}$ is an $R/I^{m(j+1)}$-module for any $j>0$, and thus $R/I_j$ is finitely filtered by $R/I$-modules. Then $H^n(\mathbf{x}^j;M) \simeq \ext_R^n(R/I_j,M) = 0$ by Lemma~\ref{L:vanish}. Using (\ref{E:ltensor}), we infer $\check{H}^n(I;M)=0$. This proves one implication.

			To prove the other implication, assume that $\check{H}^n(I;M)=0$. It is enough to show that the direct system $(H^n(\mathbf{x}^j;M) \mid j>0)$ consists of monomorphisms, because then $(\ref{E:ltensor})$ immediately yields $H^n(I;M)=H^n(\mathbf{x}^1;M)=0$. Since $H^0(\mathbf{x}^j;M)=\{m \in M \mid I_jm=0\}$, where $I_j$ is the ideal generated by $\mathbf{x}^j$, we infer that the directed system $(H^0(\mathbf{x}^j;M) \mid j >0)$ consists of inclusions. For $n>0$, we argue as follows. By Corollary~\ref{C:Koszul-vanish}, we have $M \in \bigcap_{i=0}^{n-1}\Ker \ext_R^i(R/I,-)$. Consider the long exact sequence obtained by applying $\Hom_R(-,M)$ to the following exact sequence, where $\pi$ is the natural projection:
			$$0 \rightarrow I_j/I_{j+1} \rightarrow R/I_{j+1} \xrightarrow{\pi} R/I_j \rightarrow 0.$$
			By the same argument using Lemma~\ref{L:vanish} as above, we have
			$$\ext_R^{n-1}(I_j/I_{j+1},M)=0.$$ 
			It follows that the map $\ext_R^n(\pi,M)$ is a monomorphism. From the construction and naturality of $q^n_M$, we infer that the there is a commutative diagram
	\begin{equation}
			\label{E:cdmono}
		\begin{tikzcd}
				\ext_R^n(R/I_j,R) \arrow{r}{q^n_M(j)}\arrow{d}{{\ext_R^n(\pi,M)}} & H^n(\mathbf{x}^j;M)\arrow{d}{\phi_j}  \\
				\ext_R^n(R/I_{j+1},R) \arrow{r}{q^n_M(j+1)} & H^n(\mathbf{x}_{j+1};M), 
		\end{tikzcd}
	\end{equation}
			where $\phi_j$ is the $j$-th map from the direct system $(H^n(\mathbf{x}^j;M) \mid j >0)$. Since the horizontal maps of (\ref{E:cdmono}) are isomorphisms, we finally infer that this direct system consists of monomorphisms, as desired.
	\end{proof}
	
	Now we treat the \emph{\v{C}ech homology}, which we define, following \cite{Schenz}, in this way:
	$$\check{H}_i(I;-)=H_{i}(\mathbb{R}\Hom_R(\check{C}^\bullet(I),-)).$$ 
	Because this functor \emph{a priori} inhabitates strictly the derived category, we would like to replace $\check{C}^\bullet(I)$ by its projective resolution, in a way that respects the limit construction of $\check{C}^\bullet(I)$. To this end, we recall the \emph{telescope cochain complex} (here we follow \cite{PSY}). For any subset $X$ of $\omega$, let $F[X]$ be the free $R$-module with basis $\{\delta_j \mid j \in X\}$. Given an element $x \in R$ we let
	$$\Tel(x) = ( \cdots \rightarrow 0 \rightarrow F[\omega] \xrightarrow{d} F[\omega] \rightarrow 0 \rightarrow \cdots),$$	
	be the cochain complex concentrated in (cohomological) degrees 0 and 1, where the differential $d$ is defined on the above basis as follows
	$$d(\delta_j) = \begin{cases}
			\delta_0, & \text{if $j=0$}, \\
			\delta_{j-1} - x\delta_j,  & \text{otherwise}.
	\end{cases}$$
	For any $j>0$, we let
	$$\Tel_j(x)=( \cdots \rightarrow 0 \rightarrow F[j] \xrightarrow{d} F[j] \rightarrow 0 \rightarrow \cdots)$$
	be the subcomplex of $\Tel(x)$, so that $\Tel(x)=\bigcup_{j>0}\Tel_j(x)$. More generally, given a sequence of elements $\mathbf{x}=(x_1,x_2,\ldots,x_n)$ of $R$, we let
	$$\Tel_j(\mathbf{x}) = \bigotimes_{i=1}^n \Tel_j(x_i) \quad\textrm{and}\quad \Tel(\mathbf{x}) = \bigotimes_{i=1}^n \Tel(x_i).$$
	Note that again $\Tel(\mathbf{x}) = \bigcup_{j>0} \Tel_j(\mathbf{x})$. It follows from (\cite[Lemma 5.7]{PSY}) that there are natural homotopy equivalences 
	$$w_{\mathbf{x},j}: \Tel_j(\mathbf{x}) \rightarrow K_\bullet(\mathbf{x}^j)^*,$$
	such that their limit map 
	$$w_\mathbf{x}: \Tel(\mathbf{x}) \rightarrow \check{C}^\bullet(\mathbf{x})$$
	is a quasi-isomorphism. If $I$ is the ideal generated by $\mathbf{x}$, we can now represent the \v{C}ech homology as follows:
	\begin{gather*}
			\label{E:rhom}
			\check{H}_i(I;M) = H_{i}(\mathbb{R}\Hom_R(\check{C}^\bullet(I),M)) \simeq  H_{i}(\Hom_R(\Tel(\mathbf{x}),M)) \simeq \\
			\simeq H_{i}(\Hom_R(\varinjlim{}_{j>0}\Tel_j(\mathbf{x}),M)) \simeq H_{i}(\varprojlim{}_{j>0}\Hom_R(\Tel_j(\mathbf{x}),M)) \simeq \\
			\simeq H_{i}(\varprojlim{}_{j>0}(\Tel_j(\mathbf{x})^* \otimes_R M)).
	\end{gather*}
Of course, in general, taking homology does not commute with inverse limits. On the other hand, the inverse system $( \Tel_j(\mathbf{x})^* \otimes_R M \mid j>0)$ consists of epimorphisms, and thus satisfies the Mittag-Leffler condition (see \cite[Definition 3.5 and Lemma 3.6]{GT}). Using \cite[Theorem 3.5.8]{We}, we have for each $i \geq 0$ the following exact sequence:
$$
		0 \rightarrow \varprojlim{}_{j>0}^1 H_{i+1}(\Tel_j(\mathbf{x})^* \otimes_R M) \rightarrow \check{H}_i(I;M) \rightarrow \varprojlim{}_{j>0} H_i(\Tel_j(\mathbf{x})^* \otimes_R M) \rightarrow 0.
$$
Here, the symbol $\varprojlim_{j>0}^1$ stands for the first right derived functor of the inverse limit functor $\varprojlim_{j>0}$.
Furthermore, because
$$w_{\mathbf{x},j}^*:  K_\bullet(\mathbf{x}^j) \rightarrow \Tel_j(\mathbf{x})^*$$
is also a natural homotopy equivalence of complexes, we can rewrite this sequence as:
\begin{equation}
	\label{E: limone}
		0 \rightarrow \varprojlim{}_{j>0}^1 H_{i+1}(\mathbf{x}^j;M) \rightarrow \check{H}_i(I;M) \rightarrow \varprojlim{}_{j>0} H_i(\mathbf{x}^j;M) \rightarrow 0.
\end{equation}
We are ready to prove that \v{C}ech complexes allow to classify tilting classes (see also Theorem~\ref{T:mainthm3} below).
\begin{lem}
		\label{lem:Cech-hom}
		Let $R$ be a commutative ring, $I$ a finitely generated ideal, and $n > 0$. Then
			$$\bigcap_{i=0}^{n-1}\Ker H_i(I;-) = \bigcap_{i=0}^{n-1}\Ker \check{H}_{i}(I;-).$$
\end{lem}
\begin{proof}
		We proceed again by induction on $n \geq 0$, and fix throughout the proof a module $M \in \bigcap_{i=0}^{n-1}\Ker H_i(I;-) = \bigcap_{i=0}^{n-1}\Ker \check{H}_{i}(I;-)$. 

		Suppose first that $H_n(I;M)=0$. By Remark~\ref{R:Koszul-tor}, $H_{i}(\mathbf{x}^j;M)$ is naturally isomorphic to $\tor_{i}^R(R/I_j,M)$ for each $j>0$ and $i=0,1,\ldots,n+1$. An argumentation analogous to the one in the proof of Lemma~\ref{lem:Cech-cohom} then yields that $H_n(\mathbf{x}^j,M)=0$ for each $j>0$, and that the inverse system $(H_{n+1}(\mathbf{x}^j;M) \mid j>0)$ consists of epimorphisms. In particular, this system is Mittag-Leffler, and thus \cite[Lemma 3.6]{GT} implies that
	$$\varprojlim{}_{j>0}^1 H_{n+1}(\mathbf{x}^j;M)=0.$$
	Therefore, we can use (\ref{E: limone}) to infer that $\check{H}_n(I;M) \simeq \varprojlim_{j>0} H_n(\mathbf{x}^j;M)=0$. This proves $\bigcap_{i=0}^{n-1}\Ker H_i(I;-) \subseteq \bigcap_{i=0}^{n-1}\Ker \check{H}_{i}(I;-)$.

	To prove the other inclusion, suppose that $\check{H}_n(I;M)=0$. By (\ref{E: limone}), this implies $\varprojlim_{j>0} H_n(\mathbf{x}^j;M)=0$. Using again the same argument as above for homological degree shifted by $-1$, the inverse system $(H_{n}(\mathbf{x}^j;M) \mid j>0)$ consists of epimorphisms (in the initial case of $n=0$, it consists of projections $M/I_{j+1}M \rightarrow M/I_jM$). It follows that $H_n(\mathbf{x};M)=0$, and thus $H_n(I;M)=0$.
\end{proof}

\subsection{Main theorem revisited}

In this section, we show that instead of Ext/Tor or Koszul (co)homology, we can use either local (co)homology, or \v{C}ech (co)homology, in the formulation of Theorem~\ref{T:mainthm2}. We prove the remaining parts in the following Lemmas, and then state the alternative classification Theorem.
\begin{lem}
		\label{L:local-cohom-cotilt}
		Let $R$ be a commutative ring and $\mathfrak{S}=(\mathcal{F}_0,\mathcal{F}_1,\ldots,\mathcal{F}_{n-1})$ a characteristic sequence of length $n$. Then
		$$\mathcal{C}(\mathfrak{S})=\bigcap_{i=0}^{n-1} \bigcap_{I \in \mathcal{G}_i^f(\mathfrak{S})} \{M \in \ModR \mid \mathbb{R}^i\Gamma_I(M)=0\}.$$
\end{lem}
\begin{proof}
	This follows directly from the definition of $\mathcal{C}(\mathfrak{S})$ (see Notation~\ref{C(S)}). Indeed, if $I \in \mathcal{G}_m^f$ and $0\le m<n$, then a module $M\in\mathcal{C}(\mathfrak{S})$ must be in the class 
	$$ \bigcap_{i=0}^m \Ker \Hom_R(R/I,\Omega^{-i}M) = \bigcap_{i=0}^m \Ker \Gamma_I(\Omega^{-i}M). $$
	We prove by induction on $j\le m$ that $\mathbb{R}^j\Gamma_I(M)\simeq\Gamma_I(\Omega^{-j}M)$ for each $M\in \Ker \Hom_R(R/I,\Omega^{-(j-1)}M) = \Ker \Gamma_I(\Omega^{-(j-1)}M)$ (this condition is vacuous for $j=0$).
	This follows from the definition of $\Gamma_I$ for $j=0$, and for $j>0$ note that $\Gamma_I(\Omega^{-(j-1)}M)=0$ implies $\Gamma_I(E(\Omega^{-(j-1)}M))=0$ and we have an exact sequence 
	$$ 0= \Gamma_I(E(\Omega^{-(j-1)}M)) \rightarrow \Gamma_I(\Omega^{j}M) \rightarrow \mathbb{R}^j\Gamma_I(M) \rightarrow 0. \eqno{\qedhere} $$
\end{proof}
\begin{lem}
		\label{L:local-hom-tilt}
		Let $R$ be a commutative ring, $I$ a finitely generated ideal, and $n>0$. Then:
		$$\bigcap_{i=0}^{n-1} \Ker \tor_i^R(R/I,-) = \bigcap_{i=0}^{n-1} \{M \in \ModR \mid \mathbb{L}_i \Lambda_I(M)=0\}.$$
\end{lem}
\begin{proof}
		The shape of the proof is the same as that of Lemma~\ref{lem:Cech-hom}, using the exact sequence \cite[Proposition 1.1]{GM} instead of (\ref{E: limone}). For the convenience of the reader, we provide details here. We proceed by induction on $n \geq 0$ (the case of $n=0$ is a vacuous statement). For the induction step, we will assume that $M \in \bigcap_{i=0}^{n-1} \Ker \mathbb{L}_i \Lambda_I$. By \cite[Proposition 1.1]{GM}, there is an exact sequence:
		\begin{equation}
			\label{E:GM-sequence}
				0 \rightarrow \varprojlim{}_{j>0}^1 \tor{}_{n+1}^R(R/I^j,M) \rightarrow \mathbb{L}_{n}\Lambda_I(M) \rightarrow \varprojlim_{j>0} \tor{}_{n}^R(R/I^j,M) \rightarrow 0.
		\end{equation}
		If $\mathbb{L}_{n}\Lambda_I(M)=0$, then rightmost term also vanishes. By the induction hypothesis we have $M \in \bigcap_{i=0}^{n-1} \Ker \tor_i^R(R/I,-)$. Applying $- \otimes_R M$ to the exact sequence
		$$0 \rightarrow I^j/I^{j+1} \rightarrow R/I^{j+1} \rightarrow R/I^j \rightarrow 0,$$
		and noting that $\tor_{n-1}^R(I^j/I^{j+1},M)=0$, we infer that the inverse system 
		$$(\tor{}_{n}^R(R/I^j,M) \mid j>0)$$ 
		consists of epimorphisms. This shows that $\tor_n^R(R/I,M)=0$, proving one inclusion. To prove the other inclusion, suppose now that $\tor_n^R(R/I,M)=0$. It follows easily that $\tor_n^R(R/I^j,M)=0$ for all $j>0$, and thus the rightmost term of \eqref{E:GM-sequence} is zero. By repeating the same argument as above for a homological degree shifted by $1$, we obtain that the inverse system $(\tor{}_{n+1}^R(R/I^j,M) \mid j>0)$ consists of epimorphisms, and thus is Mittag-Leffler. Hence, the leftmost term of \eqref{E:GM-sequence} also vanishes by \cite[Lemma 3.6]{GT}, and thus $\mathbb{L}_n\Lambda_I(M)=0$, finishing the proof.
	\end{proof}
	\begin{thm}
			\label{T:mainthm3} Let $R$ be a commutative ring. Consider the following collections:
			\begin{enumerate}
					\item[(i)] characteristic sequences $\mathfrak{S}=(\mathcal{F}_0,\mathcal{F}_1,\ldots,\mathcal{F}_{n-1})$ of length $n$,
					\item[(ii)] $n$-tilting classes in $\ModR$,
					\item[(iii)] $n$-cotilting classes of cofinite type in $\ModR$,
			\end{enumerate}
			The following assignments are bijections $(i) \rightarrow (ii)$:
			$$\mathfrak{S} \mapsto \bigcap_{i=0}^{n-1} \bigcap_{I \in \mathcal{G}_i^f(\mathfrak{S})} \{M \in \ModR \mid \mathbb{L}_i \Lambda_I(M)=0\},$$
			$$\mathfrak{S} \mapsto \bigcap_{i=0}^{n-1} \bigcap_{I \in \mathcal{G}_i^f(\mathfrak{S})} \{M \in \ModR \mid \check{H}_i(I;M)=0\},$$
			and the following assignments are bijections $(i) \rightarrow (iii)$:
			$$\mathfrak{S} \mapsto \bigcap_{i=0}^{n-1} \bigcap_{I \in \mathcal{G}_i^f(\mathfrak{S})} \{M \in \ModR \mid \mathbb{R}_i \Gamma_I(M)=0\},$$
			$$\mathfrak{S} \mapsto \bigcap_{i=0}^{n-1} \bigcap_{I \in \mathcal{G}_i^f(\mathfrak{S})} \{M \in \ModR \mid \check{H}^i(I;M)=0\}.$$
			\end{thm}
\begin{proof}
		Follows by putting together Theorem~\ref{T:mainthm2}, Corollary~\ref{C:Koszul-vanish}, and Lemmas~\ref{L:local-cohom-cotilt}, \ref{L:local-hom-tilt}, \ref{lem:Cech-cohom}, and \ref{lem:Cech-hom}.
\end{proof}

\subsection{The big picture}
\label{subsec:picture}
    As the four homological and four cohomological theories used in the classification Theorems~\ref{T:mainthm2} and~\ref{T:mainthm3} may feel a little overwhelming, we devote this and the next subsection to a short explanation of the relations between these. In particular we show that the \v{C}ech (co)homology, analogously to the local (co)homology, also induces a pair of adjoint functors between derived categories, and that there are always comparison functors between the \v{C}ech and local (co)homologies which are equivalences under a technical condition.
    
    Although in this material is not really new, it requires some effort to extract it from the existing literature~\cite{AJL,GM,PSY,Schenz}. Here we especially rely on a recent and original presentation in \cite[Theorem 3.4]{Pos}. In fact, all we want to say is essentially a reformulation of \cite[Theorem 3.4]{Pos} and its proof, and we refer the reader to \cite{Pos} for a more comprehensive treatment. We will freely use the theory of localization of the derived category, as well as recollements and their translation to TTF triples. For these we refer  to \cite{BBD}, \cite{Kr}, or \cite{NS}.

 Given a full subcategory $\mathcal{A}$ of $\ModR$, let us denote by $\Der_\mathcal{A}(\ModR)$ the subcategory of $\Der(\ModR)$ consisting of all complexes such that their homology modules live in $\mathcal{A}$. If $\mathcal{A}$ is an extension closed abelian subcategory of $\ModR$, then $\Der_\mathcal{A}(\ModR)$ is a triangulated subcategory of $\Der(\ModR)$.
 Given a set of objects $\mathcal{S}$ in $\Der(\ModR)$, let $\Loc(\mathcal{S})$ be the smallest localizing subcategory of $\Der(\ModR)$ containing $\mathcal{S}$. Let $I$ be an ideal generated by $\mathbf{x}=(x_1,x_2,\ldots,x_n)$. By \cite[Proposition 5.1]{Pos}, the category $\Der_{\mathcal{T}_I}(\ModR)$ is generated (as a triangulated subcategory) by the compact object $\Tel_j(\mathbf{x})$ for any fixed $j>0$. Using \cite[Proposition 2.1.2]{KP} (see also~\cite[\S6]{DG02}), we have
	$$\Loc(R/I)=\Loc(K_\bullet(I))=\Der{}_{\mathcal{T}_I}(\ModR).$$
Using the machinery of localization theory of triangulated categories (see \cite{Kr}, namely \cite[4.9.1, 4.13.1, 5.3.1, 5.4.1, 5.5.1]{Kr}) there is a recollement (we adopt the convention that going up amounts to taking \emph{left adjoints})
	\begin{equation}
		\label{E:recollement}
	\begin{tikzcd}
			\Der_{\mathcal{T}_I}(\ModR)^\perp \arrow[yshift=0ex]{r}{\subseteq} & \Der(\ModR) \arrow[yshift=0]{r}[pos=.6]{T} \arrow[bend left=25]{l}[above]{}\arrow[bend right=25]{l}[above]{}
& \Der_{\mathcal{T}_I}(\ModR) \arrow[bend left=25]{l}[below]{H}\arrow[bend right=25]{l}[above]{\iota},
	\end{tikzcd}
	\end{equation}
	corresponding (as in \cite[2.1]{NS}) to the \emph{TTF triple}
	$$(\Loc(R/I),\mathcal{Y},\mathcal{Z}),$$
	where $\mathcal{Y}=\Der_{\mathcal{T}_I}(\ModR)^\perp$. By \cite[Theorem 2.2.4]{KP}, we also have 
	\begin{equation}
		\label{E:loc-KP}
	\mathcal{Y}=\Loc(R_{x_1},R_{x_2},\ldots,R_{x_n}).
	\end{equation}
	In particular it is a localizing subcategory, and hence a \emph{tensor ideal} by \cite[Lemma 1.1.8]{KP}.
	Consider the triangle
	\begin{equation}
	\label{E:cech-triangle}
	\check{C}^\bullet(I) \xrightarrow{f} R \rightarrow \Cone(f) \rightarrow \Sigma \check{C}^\bullet(I),
	\end{equation}
	induced by the identity map $R \rightarrow R$ in degree 0. Let $M$ be a complex and apply $- \otimes_R^{\mathbb{L}}M$ to \eqref{E:cech-triangle} in order to obtain a triangle
	\begin{equation}
		\label{E:approxtriangle}
\check{C}^\bullet(I) \otimes_R M \xrightarrow{f \otimes_R M} M \longrightarrow \Cone(f) \otimes_R M \longrightarrow \Sigma \check{C}^\bullet(I) \otimes_R M.
	\end{equation}
    Then $\check{C}^\bullet(I) \otimes_R M \in \Der_{\mathcal{T}_I}(\ModR)$ by \cite[Lemma 1.1]{Pos}. Note that $\Cone(f)$ is quasi-isomorphic to the complex
	$$\bigoplus_{i=1}^n R_{x_i} \rightarrow \bigoplus_{1 \leq i < j \leq n} R_{x_i,x_j} \rightarrow \cdots \rightarrow R_{x_1,x_2,\ldots,x_n}.$$
	Since $\mathcal{Y}$ is thick and a tensor ideal, it follows from \ref{E:loc-KP} that $\Cone(f) \otimes_R M \in \mathcal{Y}$. Then (\ref{E:approxtriangle}) is the approximation triangle for $M$ with respect to the torsion pair $(\Loc(R/I),\mathcal{Y})$, and thus by \cite[Proposition 1.3.3]{BBD} or \cite[Theorem 1.1.9]{KP}, the right adjoint $T$ is equivalent to $\check{C}^\bullet(I) \otimes_R -$.

	Since $T$ composed with the inclusion $\iota$ is equivalent to $\check{C}^\bullet(I) \otimes_R^{\mathbb{L}} -$, passing to right adjoints we obtain $HT \simeq \mathbb{R}\Hom_R(\check{C}^\bullet(I),-)$. From the description of recollements arising from TTF triples (see \cite[2.1]{NS}), we get 
	$$\mathbb{R}\Hom_R(\check{C}^\bullet(I),-) \simeq HT \simeq \tau^\mathcal{Z} \iota T,$$
	where $\tau^\mathcal{Z}$ is the left adjoint to the inclusion $\mathcal{Z} \subseteq \Der(\ModR)$. Since $\iota T$ is a triangle equivalence $\Loc(R/I) \rightarrow \mathcal{Z}$, we finally infer that $\mathbb{R}\Hom_R(\check{C}^\bullet(I),-)$ is the left adjoint to the inclusion $\mathcal{Z} \subseteq \Der(\ModR)$. Now it can be easily checked that the triangle obtained by applying $\mathbb{R}\Hom_R(\check{C}^\bullet(I),-)$ to \eqref{E:cech-triangle} is the approximation triangle with respect to the torsion pair $(\mathcal{Y},\mathcal{Z})$.
	This yields that a complex $M$ is in $\mathcal{Z}$ if and only if the natural map $\mathbb{R}\Hom_R(\check{C}^\bullet(I),f): M \mapsto \mathbb{R}\Hom_R(\check{C}^\bullet(I),M)$ is an isomorphism. Combining \cite[Lemma 2.2 a),c)]{Pos}, and the proof of \cite[Theorem 3.4]{Pos}, we conclude that $\mathcal{Z}=\Der_{\mathcal{C}_I}(\ModR)$, where $\mathcal{C}_I$ is the subcategory of $I$-contramodules (see~\S\ref{subsec:loccohom}).

	To summarize, since a composition of right (left) adjoints is a right (left) adjoint, respectively, we have two compositions of adjoint pairs depicted in~\eqref{E:pic-tor} and~\eqref{E:pic-contra}. Here $F := \mathbb{R}\Gamma_I{\restriction \Der_{\mathcal{T}_I}(\ModR)}$ is the right adjoint of the canonical functor $\Der(\mathcal{T}_I) \to \Der_{\mathcal{T}_I}(\ModR)$ and similarly $G$ is the left adjoint of the canonical functor $\Der(\mathcal{C}_I) \to \Der_{\mathcal{C}_I}(\ModR)$. Both $\mathbb{R}\Gamma_I$ and $\mathbb{L}\Delta_I$ are then naturally equivalent to the compositions of the corresponding two ``short'' adjoints pointing to the left:
	\begin{equation}
	\label{E:pic-tor}
	\begin{tikzcd}
			\Der(\mathcal{T}_I) \arrow[bend left]{r} & \Der_{\mathcal{T}_I}(\ModR) \arrow[bend left]{l}{F}
			\arrow[bend left]{r}{\subseteq} & \Der(\ModR)\arrow[bend left]{l}[pos=0.51]{\check{C}^\bullet(I) \otimes_R -} \arrow[bend left=65]{ll}{\mathbb{R}\Gamma_I}
	\end{tikzcd}
	\end{equation}
	\begin{equation}
	\label{E:pic-contra}
	\begin{tikzcd}
			\Der(\mathcal{C}_I) \arrow[bend right]{r} & \Der_{\mathcal{C}_I}(\ModR) \arrow[bend right]{l}[above]{G}
			\arrow[bend right]{r}[below]{\subseteq} & \Der(\ModR)\arrow[bend right]{l}[above,pos=0.6]{\mathbb{R}\Hom_R(\check{C}^\bullet(I),-)}\arrow[bend right=65]{ll}{\mathbb{L}\Delta_I}
	\end{tikzcd}
	\end{equation}
	
	There is a technical condition on $I$, so-called weak proregularity of $I$, which ensures (and in fact is equivalent to) that both adjunctions on the left in \eqref{E:pic-tor} and \eqref{E:pic-contra} are in fact equivalences (see \cite[Theorem 1.3, Corollary 2.10]{Pos}). We will discuss this in \S\ref{subsec:weak-proregularity}
	
	Here we conclude by noting that in such a case, the local (co)homology coincides with the \v{C}ech (co)homology. Furthermore, weak proregularity also implies that $\mathbb{L}\Delta_I$ is naturally equivalent to $\mathbb{L}\Lambda_I$ (\cite[Lemma 2.5]{Pos}) and, since $\Der_{\mathcal{T}_I}(\ModR)$ and $\Der_{\mathcal{C}_I}(\ModR)$ are always equivalent (cf.\ the recollement~\eqref{E:recollement} or \cite[Theorem 3.4]{Pos}), weak proregularity also implies that: 
	$$\Der(\mathcal{T}_I) \simeq \Der(\mathcal{C}_I).$$
	The latter statement is known as the \emph{Matlis-Greenlees-May duality}, \cite{DG02,PSY,Pos}.

\subsection{Weak proregularity}
\label{subsec:weak-proregularity}
	A classical result (\cite{EGA3}) says that, over a commutative noetherian ring, the local cohomology coincides with the \v{C}ech cohomology. The dual result for the left derived completion functor and \v{C}ech homology (\cite{GM,PSY,Schenz,AJL}) is a much more recent development. However, over a general commutative ring, the local and \v{C}ech (co)homologies need not be the same, despite the fact that we can use either of them in Theorem~\ref{T:mainthm3} to classify (co)tilting classes. Here we gather relevant results to understand the issue.
	
\begin{definition}[{\cite[\S2]{Schenz}}] ~
			\begin{enumerate}
				\item An inverse system $(M_i, f_{ji} \mid j \geq i)$ of modules is \emph{pro-zero} if for every $i$ there is $j \geq i$ such that $f_{ji}$ is zero.
				\item Let $R$ be $\mathbf{x}=(x_1,x_2,\ldots,x_n)$ be a sequence of elements of a commutative ring $R$. We say that $\mathbf{x}$ is \emph{weakly proregular} if the inverse system $(H_i(\mathbf{x}^j;R) \mid j>0)$ (see \S\ref{subsec:loccohom}) is pro-zero for each $i>0$.
			\end{enumerate}
	\end{definition}
\begin{fact}(\cite[Corollary 6.2]{PSY})
		The weak proregularity of $\mathbf{x}$ depends only on the ideal $I$ generated by $\mathbf{x}$ (in fact only on $\sqrt{I}$). This legitimizes us to define a finitely generated ideal $I$ to be \emph{weakly proregular}, if any of its finite generating sequences is weakly proregular.
\end{fact}
	If $R$ is noetherian, then any ideal is weakly proregular (\cite[Theorem 4.34]{PSY}). On the other hand, over a general commutative rings, there can easily be non-weakly proregular finitely generated ideals (\cite[Example 1.4]{GM}). It turns out that this property characterizes precisely when the local (co)homology of an ideal coincides with the \v{C}ech (co)homology. The result on the side of cohomology is \cite[Theorem 3.2]{Schenz}. In the same paper, the homology analog is proved under the extra assumption that each element of the generator sequence forms itself a one-element weakly proregular sequence (\cite[Theorem 4.5]{Schenz}). This extra assumption was removed in \cite{PSY}.
	\begin{thm}
			\label{T:wp-cohom}
			Let $R$ be a commutative ring and $I$ a finitely generated ideal. Then the following are equivalent:
			\begin{enumerate}
				\item[(a)] $I$ is weakly proregular,
				\item[(b)] the functorial map $$\mathbb{R}\Gamma_I(M) \rightarrow \check{C}^\bullet(I) \otimes_R M$$ is a quasi-isomorphism for each module $M$,
				\item[(c)] the functorial map $$\mathbb{R}\Hom_R(\check{C}^\bullet(I),-) \rightarrow \mathbb{L}_i\Lambda_I(M)$$ is a quasi-isomorphism for each module $M$.
			\end{enumerate}
	\end{thm}
	\begin{proof}
			Although all the difficult steps have already been carried out by the aforementioned authors, we need to make a few explanations to establish the full equivalence in this form.

			The equivalence $(a) \leftrightarrow (b)$ is a slight reformulation of the one in \cite{Schenz}---our statement is a weakened version of statement (iii) of \cite[Theorem 3.2]{Schenz}, which easily implies (ii), and thus also (i) by the proof.

			The equivalence $(a) \leftrightarrow (c)$ is explained as follows. The implication $(a) \rightarrow (c)$ is proved in \cite[Corollary 5.25]{PSY}. The converse implication follows from the proof of \cite[Theorem 4.5]{Schenz}. Indeed, both the implications $(iii) \rightarrow (iv)$ and $(iv) \rightarrow (i)$ of \cite[Theorem 4.5]{Schenz} do not use (or need) the assumption of ``bounded torsion'' imposed in the statement of \cite[Theorem 4.5]{Schenz}. 
	\end{proof}
	Therefore, any example of a non-weakly proregular ideal $I$ (such as \cite[Example 1.4]{GM}) gives a situation where both the local cohomology and homology are not computed via the \v{C}ech complex. Indeed, not only the functorial map from Theorem~\ref{T:wp-cohom} fails to be a quasi-isomorphism, but inspecting the proofs in \cite{Schenz}, some flat (injective) module has a non-zero higher \v{C}ech homology (cohomology), but the local homology (cohomology) will vanish, respectively.

\section{Construction of the corresponding cotilting modules}
\label{sec:construct}
In this section we construct to each cotilting class of cofinite type over a commutative ring a cotilting module inducing it. The construction generalizes ideas from \cite{HST}.
\begin{lem}
	\label{L98}
	Let $R$ be a ring and $\mathcal{C}$ a $n$-cotilting class in $\RMod$. Suppose that $C$ is a left $R$-module satisfying ${}^\perp C = \mathcal{C}$, $C \in \mathcal{C}$, and $\mathcal{C} \subseteq \cogen(C)$. Then $C$ is a cotilting module.
\end{lem}
\begin{proof}
		We prove that $\cogen_{n} (C) = \mathcal{C}$, which implies that $C$ is a $n$-cotilting module by \cite[Theorem 3.11]{B}. Since $\mathcal{C}$ is a $n$-cotilting class and $C \in \mathcal{C}$, we have inclusion $\cogen_{n} (C) \subseteq \mathcal{C}$. To show the other inclusion, let $M \in \mathcal{C}$, put $I=\Hom_R(M,C)$, and let $\varphi: M \rightarrow C^I$ be the coevaluation map. Since $M \in \cogen(C)$, we have that the map $\varphi$ is injective. Applying $\Hom_R(-,C)$ onto the exact sequence $0 \rightarrow M \xrightarrow{\varphi} C^I \rightarrow X \rightarrow 0$ yields
	$$\Hom{}_R(C^I,C) \xrightarrow{\Hom_R(\varphi,C)} \Hom{}_R(M,C) \rightarrow \ext{}_R^1(X, C) \rightarrow \ext{}_R^1(C^I,C) = 0.$$
	As $\Hom_R(\varphi,C)$ is clearly surjective, we have that $\ext_R^1(X, C)=0$. Using the fact that $C^I,M \in \mathcal{C} = {}^\perp C$, we infer that $X \in \mathcal{C} = {}^\perp C$. Repeating this argument shows that indeed $M \in \cogen_{n} (C)$. 
\end{proof}
\begin{cor}
		Let $R$ be a commutative ring and $C$ a cotilting module such that the induced cotilting class $\mathcal{C}={}^\perp C$ is of cofinite type. Let $\mathfrak{S}$ be the characteristic sequence such that $\mathcal{C}=\mathcal{C}(\mathfrak{S})$. Let $W_j$ be an injective module such that $\cogen(W_j)=\mathcal{F}_j(\mathfrak{S})$. Then $\Omega^{-j}(C) \oplus W_j$ is a cotilting module inducing the cotilting class $\mathcal{C}_{(j)}$.
\end{cor}
\begin{proof}
		Put $C'=\Omega^{-j} C \oplus W_j$. We clearly have ${}^\perp C' = \mathcal{C}_{(j)}$, Lemma~\ref{L31} gives $C' \in \mathcal{C}_{(j)}$, and $\mathcal{C}_{(j)} \subseteq \mathcal{F}_j(\mathfrak{S}) = \cogen(C')$. Therefore, Lemma~\ref{L98} implies that $C'$ is a cotilting module.
\end{proof}
In the rest of the section let $R$ be a commutative ring and let $\mathfrak{S}=(\mathcal{F}_0,\mathcal{F}_1,\ldots,\mathcal{F}_{n-1})$ be a characteristic sequence. Our goal is to construct a cotilting module $C(\mathfrak{S})$ such that ${}^\perp C(\mathfrak{S})=\mathcal{C}(\mathfrak{S})$.
\begin{construction}
		\label{C99}
	We aim to construct a (co)complex of injective modules
	$$0 \rightarrow E^0 \xrightarrow{\psi_0} E^1 \xrightarrow{\psi_1} \cdots \xrightarrow{\psi_{n-1}} E^n \xrightarrow{\psi_n} 0 (=E^{n+1}),$$
	where $E^i$ is in cohomological degree $i$, satisfying the following properties:
	\begin{enumerate}
		\item[(i)] the cohomology of the complex vanishes with the exception of degree 0,
		\item[(ii)] for each $i=0,1,\ldots,n$, the kernel $C^i$ of $\psi_i$ is a $(n-i)$-cotilting module such that ${}^\perp C^i=\mathcal{C}_{(i)}$.
	\end{enumerate}
	We construct the complex by backwards induction on $i=n,n-1,\ldots,1,0$. For the step $i=n$, we let $E^n$ be an injective cogenerator of $\ModR$, and $\psi_n$ be the zero map.

	Suppose that we have already constructed the complex down to degree $k+1$ for some $0\le k<n$ so that it is exact in degrees $>k+1$ and satisfies (ii). By \cite[Theorem 15.9]{GT}, there is a $\mathcal{C}_{(k)}$-cover $f\colon F^k \rightarrow C^{k+1}$ of $C^{k+1}$. 
	\begin{lem}
		\label{L41}
		The module $F^k$ is injective.
	\end{lem}
	\begin{proof}
			Because $F^k \in \mathcal{C}_{(k)}$, we have by Lemma~\ref{L31}(ii) that $\Omega^{-1}F^k \in \mathcal{C}_{(k+1)}$. By the inductive premise, $C^{k+1}$ is a cotilting module such that ${}^\perp C^{k+1}=\mathcal{C}_{(k+1)}$, and therefore $C^{k+1} \in \mathcal{C}_{(k+1)}^\perp$. It follows that the cover $f\colon F^k \rightarrow C^{k+1}$ can be extended to a map $f'\colon E(F^k) \rightarrow C^{k+1}$. As $E(F^k) \in \mathcal{C}_{(k)}$ by Proposition~\ref{P01}, it can be easily seen that $f'$ is an $\mathcal{C}_{(k)}$-precover of $C^{k+1}$. Therefore, $F^k$ is a direct summand of $E(F^k)$, proving that $F^k$ is injective.
	\end{proof}
	Now we let $E^k = F^k \oplus W^k$, where $W^k$ is an injective module such that $\cogen(W^k)=\mathcal{F}_k$ (e.g. $W^k=\prod \{E(R/J) \mid R/J \in \mathcal{F}_k\}$). We define $\psi_k: E^k \rightarrow E^{k+1}$ by setting $\psi_k{\restriction F^k} = f$ and $\psi_k{\restriction W^k} = 0$. Since $W^k$ is injective and belongs to $\mathcal{F}_k$, it is in $\mathcal{C}_{(k)}$, and then it easily follows that $\psi_k: E^k \rightarrow C^{k+1}$ is a $\mathcal{C}_{(k)}$-precover of $C^{k+1}=\Ker(\psi_{k+1})=\im(\psi_k)$.

	\begin{lem}
			\label{L42}
			We have ${}^\perp C^k = \mathcal{C}_{(k)}$ and $C^k \in \mathcal{C}_{(k)}$.
	\end{lem}
	\begin{proof}
	By the inductive premise of the construction, we know ${}^\perp \Omega^{-1} C^k = \mathcal{C}_{(k+1)}$, and thus $\bigcap_{j>1} \Ker \ext_R^j(-,C^k) = \mathcal{C}_{(k+1)}$. Let $M \in \mathcal{C}_{(k+1)}$. Then $\ext_R^1(M,C^k)=0$ if and only if any map $g\colon M \rightarrow C^{k+1}$ can be factorized through $\psi_k\colon E^k \rightarrow C^{k+1}$. If $M \in \mathcal{C}_{(k)}$, then this factorization is always possible, because $\psi_k$ is an $\mathcal{C}_{(k)}$-precover. This proves that $\mathcal{C}_{(k)} \subseteq {}^\perp C^k$.
	
	For the converse inclusion, suppose that $M \not\in \mathcal{C}_{(k)}$. Consider first the case where even $M \not\in \mathcal{F}_k$ and let $T$ be the torsion part of $M$ with respect to the torsion pair $(\mathcal{T}_k,\mathcal{F}_k)$.
	By construction, $C^{k+1}$ has an injective direct summand $W^{k+1}$ which cogenerates the torsion-free class $\mathcal{F}_{k+1}$.
	Since $M$ and also $T$ belong to $\mathcal{F}_{k+1}$, there is a non-zero map $T \rightarrow W^{k+1}$, which extends to a map $g\colon M \rightarrow W^{k+1} \subseteq C^{k+1}$. Such map does not factor through $\psi_k$ since $\Hom_R(T,E^k)=0$. The remaining case is where $M \in \mathcal{F}_k$, but $\Omega^{-1} M \not\in \mathcal{C}_{(k+1)}$. Consider the long exact sequence obtained by applying $\Hom_R(-,C^k)$ to 
	$$0 \rightarrow M \rightarrow E(M) \rightarrow \Omega^{-1} M \rightarrow 0.$$
	Since $M \in \mathcal{F}_k$, we have $E(M) \in \mathcal{C}_{(k)}$, and thus we obtain isomorphisms $\ext_R^i(M,C^k) \simeq \ext_R^{i+1}(\Omega^{-1} M,C^k)$ for all $i>0$. Hence, if $\Omega^{-1}M \not\in \mathcal{C}_{(k+1)}$, then $M \not\in {}^\perp C^k$. This finishes the proof of ${}^\perp C^k = \mathcal{C}_{(k)}$.
	
	Finally, that $C^k \in \mathcal{C}_{(k)}$ follows from Lemma~\ref{L31} and from $\Omega^{-j} C^k$ being a direct summand of $C^{k+j}$ for each $j=1,2,\ldots,n-k-1$.
	\end{proof}
	\begin{lem}
			\label{L43}
			The module $C^k$ is an $(n-k)$-cotilting module.
	\end{lem}
	\begin{proof}
			Since $\cogen(C^k)=\mathcal{F}_k \supseteq \mathcal{C}_{(k)}$, and by Lemma~\ref{L42}, the module $C^k$ and the $(n-k)$-cotilting class $\mathcal{C}_{(k)}$ satisfy the hypothesis of Lemma~\ref{L98}. Therefore, $C^k$ is $(n-k)$-cotilting by that lemma.
	\end{proof}
	This concludes the inductive step, and hence also the construction.
	\begin{notation}
	We put $C(\mathfrak{S}) = \Ker(\psi_0)$, and conclude:
	\end{notation}
	\begin{thm} \label{thmconstruction}
			Let $R$ be a commutative ring. Then the set 
			$$\{C(\mathfrak{S}) \mid \mathfrak{S} \text{ a characteristic sequence of length $n$}\}$$ 
			parametrizes the equivalences classes of all $n$-cotilting modules of cofinite type in $\ModR$.
	 \end{thm}
\end{construction}
\section{Examples of cotilting classes not of cofinite type}
\label{sec:notcofinite}
We conclude the paper by providing intriguing examples of cotilting classes which are not of cofinite type, but which are in some sense difficult to tell apart from classes of cofinite type. To explain the issue, we first give a characterization of cotilting classes of cofinite type which follows from the proof of Theorem~\ref{T01}.

\begin{thm}
		\label{thmcofinitetype}
		Let $R$ be a commutative ring and $\mathcal{C}$ an $n$-cotilting class in $\ModR$. Then $\mathcal{C}$ is of cofinite type if and only if $\mathcal{C}_{(i)}$ is closed under injective envelopes for all $i=0,1,\ldots,n-1$.
\end{thm}
\begin{proof}
		If $\mathcal{C}$ is of cofinite type, then $\mathcal{C}_{(i)}$ is easily seen to be of cofinite type too for any $i=0,1,\ldots,n-1$. Then $\mathcal{C}_{(i)}$ is closed under injective envelopes by Proposition~\ref{P01}.

		The other implication follows from Lemma~\ref{L01}, the proof of Lemma~\ref{L02}, and Lemma~\ref{L28}.
\end{proof}

In the rest of the section, we exhibit for each $n \ge 2$ a concrete example of an $n$-cotilting class which is \emph{not} of cofinite type, but for which $\mathcal{C}_{(i)}$ is closed under injective envelopes for all $i=0,1,\ldots,n-2$. To this end, we first recall a characterization of cotilting classes which is valid for any (even non-commutative) ring:

\begin{prop}
		\label{P02}
		\emph{(\cite[Proposition 3.14]{APST})}
			Let $R$ be a ring, $n \geq 0$, and $\mathcal{C}$ a class of left $R$-modules. Then $\mathcal{C}$ is $n$-cotilting if and only if the following conditions hold:
			\begin{enumerate}
				\item[(i)] $\mathcal{C}$ is definable,
				\item[(ii)] $R \in \mathcal{C}$ and $\mathcal{C}$ is closed under extensions and syzygies,
			\item[(iii)] each $n$-th syzygy module is in $\mathcal{C}$,
			\end{enumerate}
			In particular, a class of left $R$-modules is $1$-cotilting precisely when it is a definable torsion-free class containing $R$.
\end{prop}

In order to construct the aforementioned examples, we need a suitable family of examples of non-cofinite type 1-cotilting classes to start with. The following is a generalization of \cite[Proposition 4.5]{B2}:

\begin{ex}
		\label{ex:example-non-cofinite}
		Let $R$ be a local commutative ring admitting a non-trivial idempotent ideal $J$. Let $\mathcal{G}$ be a Gabriel topology of finite type such that $J \in \mathcal{G}$, and such that $\mathcal{G}$ is \emph{faithful}, i.e. $\Hom_R(R/I,R)=0$ for all $I \in \mathcal{G}$. Let $\mathcal{F}=\bigcap_{I \in \mathcal{G}} \Ker \Hom_R(R/I,-)$ be the torsion-free class of the associated hereditary torsion pair of finite type. Given a module $M$, let $\Soc_J(M)=\{m \in M \mid Jm=0\}$. We define a class $\mathcal{C}$ as follows:
$$\mathcal{C}=\{M \in \ModR \mid M/\Soc{}_J(M) \in \mathcal{F}\}.$$
Alternatively,
$$ \mathcal{C}=\{M \in \ModR \mid Jt_\mathcal{G}(M)=0\} = \{M \in \ModR \mid t_\mathcal{G}(M) \in \operatorname{Mod-}R/J\}, $$
where $t_\mathcal{G}$ is the torsion radical associated to the torsion pair $(\mathcal{T},\mathcal{F})$. 

We will show that $\mathcal{C}$ is a 1-cotilting class, but not of cofinite type. First, we claim that $\mathcal{C}=\bigcap_{I \in \mathcal{G}} \Ker \Hom_R((J+I)/I,-)$. Let $M$ be a module such that there is a non-zero map $f\colon (J+I)/I \rightarrow M$ for some $I \in \mathcal{G}$. Since $(J+I)/I$ is $J$-divisible (i.e.\ $((J+I)/I) \cdot J = (J+I)/I$), we have $\operatorname{Im}f \cap \Soc_J(M)=0$. Thus, if we compose $f$ with the surjection $M \to M/t_\mathcal{G}(M)$, we obtain a non-zero map $(J+I)/I \rightarrow M/t_\mathcal{G}(M)$. Since $(J+I)/I \in \mathcal{T}$, it follows that $M/t_\mathcal{G}(M) \not\in \mathcal{F}$ and $M \not\in \mathcal{C}$. This establishes one inclusion of the claim. 
		
		Let now $M$ be such that $\Hom_R((J+I)/I,M)=0$ for all $I \in \mathcal{G}$, and let us show that $M \in \mathcal{C}$. Towards contradiction, suppose that there is non-zero map $f\colon R/I \rightarrow M/\Soc_J(M)$ for some $I \in \mathcal{G}$. Because $\Soc_J(M/\Soc_J(M))=0$, applying $\Hom_R(-,M/\Soc_J(M))$ to $0 \to (J+I)/I \to R/I \to R/(J+I) \to 0$ yields an exact sequence 
$$ 0 \to \Hom_R(R/I,M/{\Soc}_J(M)) \to \Hom_R((J+I)/I,M/{\Soc}_J(M)).$$
Hence $f$ restricts to a non-zero map $\tilde{f}\colon (J+I)/I \rightarrow M/\Soc_J(M)$. Let us denote the image of $\tilde{f}$ by $\tilde X$ and by $X$ the full preimage of $\tilde X$ in $M$ with respect to the projection $M \rightarrow M/\Soc_J(M)$. We have $IX \subseteq \Soc_J(M)$, and thus $IJX=0$. Observe that if $JX \ne 0$, we would have a non-zero morphism $R/I \to JX$ which would similarly as above restrict to a non-zero morphism $(J+I)/I \to JX$. This is, however, not possible since we assume that $\Hom_R((J+I)/I,M)=0$. Thus $JX = 0$ and $X \subseteq \Soc_J(M)$. But then $\tilde X=0$, contradicting that $\tilde{f}$ is a non-zero map. This establishes the other inclusion of the claim. In particular, we have proved that $\mathcal{C}$ is a torsion-free class.

As a torsion-free class, $\mathcal{C}$ is closed under products and (pure) submodules. To prove that $\mathcal{C}$ is definable, it remains to treat direct limits. Since $(\mathcal{T},\mathcal{F})$ is of finite type, the torsion radical $t_\mathcal{G}$ commutes with direct limits. Given a directed system $M_i, i \in I$ with $M_i \in \mathcal{C}$, we compute: 
		$$t_\mathcal{G}(\varinjlim_i M_i) \simeq \varinjlim_i t_\mathcal{G}M_i.$$
The latter is a direct limit of $R/J$-modules, proving that $Jt_\mathcal{G}(\varinjlim_i M_i)=0$ as desired. 
Finally, since $\mathcal{G}$ is faithful, we have that $R \in \mathcal{F} \subseteq \mathcal{C}$. Using Proposition~\ref{P02}, we infer that $\mathcal{C}$ is a 1-cotilting class.

Finally, we prove that $\mathcal{C}$ is not of cofinite type. Indeed, if it was of cofinite type, Theorem~\ref{T:mainthm2} would provide us with a Gabriel topology $\mathcal{H}$ of finite type such that $\mathcal{C}=\bigcap_{K \in \mathcal{H}} \Ker \Hom_R(R/K,-)$. Since $\operatorname{Mod-}R/J \subseteq \mathcal{C}$, this implies that $J+K=R$ for all $K \in \mathcal{H}$. Since $R$ is local, the only possibility is $\mathcal{H}=\{R\}$, which forces $\mathcal{C}=\ModR$. Recall that we assumed $J \in \mathcal{G}$, and thus there is a finitely generated ideal $I \in \mathcal{G}$ with $I \subseteq J$. If $R/I \in \mathcal{C}$, then necessarily $R/I=\Soc_J(R/I)$, which implies $I=J$. But $J$ is a non-trivial idempotent ideal in a local commutative ring, so it cannot be finitely generated, a contradiction.
\end{ex}

As a next step, we would like to extend the example to $n$-cotilting classes for $n \ge 1$. The strategy is to reverse (under suitable assumptions) the process of Lemma~\ref{L01}. That is, we would like to combine a hereditary faithful torsion-free class $\mathcal{F}$ and an $n$-cotilting class, which is \emph{not} necessarily of cofinite type, to an $(n+1)$-cotilting class. Note that in Lemma~\ref{L01} we have $\operatorname{Inj-}R \cap \mathcal{F} \subseteq \mathcal{C}$, where $\operatorname{Inj-}R \subseteq \ModR$ is the class of injective modules. We will adopt this assumption for the next auxiliary result which can be viewed as an analog of~\cite[Lemma 3.10]{APST}.

\begin{lem} \label{L:naive-combination}
Let $\mathcal{F}$ be a hereditary torsion-free class and let $C \in \ModR$ be a module such that the class $\mathcal{C}={}^\perp C$ contains $\operatorname{Inj-}R \cap \mathcal{F}$. Then the class
$$\mathcal{D}=\{M \in \RMod \mid M \in \mathcal{F} \text{ and $\Omega^{-1} M \in \mathcal{C}$}\}$$
satisfies the following property: If $0 \to L \to M \to N \to 0$ is a short exact sequence with $M \in \mathcal{D}$, then $L \in \mathcal{D}$ if and only if $N \in \mathcal{C}$.
\end{lem}

\begin{proof}
First observe that $M \in \mathcal{D}$ if and only if $M\in \mathcal{F} \cap \mathcal{C}$ and each morphism $f\colon M \to C$ extends to $\tilde{f}\colon E(M) \to C$. Indeed, this follows at once from the long exact sequence obtained by applying $\Hom_R(-,C)$ to $0 \to M \to E(M) \to \Omega^{-1}(M) \to 0$.

Now let $\varepsilon\colon 0 \to L \overset{u}\to M \to N \to 0$ be exact with $M \in \mathcal{D}$. In particular $L \in \mathcal{F}$ and, if we apply $\Hom_R(-,C)$ to $\varepsilon$, we obtain isomorphisms $\ext^i_R(L,C) \cong \ext^{i+1}_R(N,C)$ for all $i \ge 1$.

Suppose next that $L \in \mathcal{D}$. In particular $L \in {}^\perp C$, so $\ext^i_R(N,C) = 0$ for all $i\ge2$. It remains to show that $\ext^1_R(N,C) = 0$. To this end, let $v\colon L \to E(L)$ be an injective envelope and let $w\colon M \to E(L)$ be a morphism such that $wu = v$, obtained by the injectivity of $E(L)$. If $f\colon L \to C$ is any homomorphism, it extends to $\tilde{f}\colon E(L) \to C$ since $L\in\mathcal{D}$. In particular $f = \tilde{f} v = \tilde{f}wu$, showing that the leftmost morphism in the following exact sequence is surjective:
$$ \Hom_R(M,C) \to \Hom_R(L,C) \to \ext_R^1(N,C) \to \ext_R^1(M,C) = 0 $$
Thus $\ext_R^1(N,C) = 0$ and $N \in \mathcal{C}$.

Conversely suppose that $N \in \mathcal{C}$. Then $L \in \mathcal{F} \cap \mathcal{C}$ by the above observations. To see that $L \in \mathcal{D}$, it remains to prove that each $f\colon L \to C$ extends to $\tilde{f}\colon E(L) \to C$. Consider the following commutative square with injective envelopes in rows, where the right vertical map is completed using the injectivity of E(M):
$$
\begin{CD}
L @>{v}>> E(L)  \\
@V{u}VV @VVV \\
M @>{z}>> E(M)
\end{CD}
$$ 
Then $\Hom_R(u,C)$ is surjective since $\ext^1_R(N,C) = 0$ and $\Hom_R(z,C)$ is surjective since $M\in\mathcal{D}$. It follows that $\Hom_R(v,C)$ is surjective, as required.
\end{proof}

Now we prove a result which, under more restrictive assumptions, allows us to combine a hereditary torsion-free class with an $n$-cotilting class.

\begin{prop} \label{P:combine}
Let $R$ be a ring and $n \ge 1$. Suppose that all of the following conditions are satisfied:
\begin{enumerate}
\item $\mathcal{C}$ is an $n$-cotilting class in $\ModR$ and $\Omega^{-1}(R) \in \mathcal{C}$.
\item $(\mathcal{T},\mathcal{F})$ is a faithful hereditary torsion pair of finite type in $\ModR$.
\item $\bigcap_{i=0}^{n-1} \Ker \ext_R^{i}(\mathcal{T},-) \subseteq \mathcal{C}$.
\end{enumerate}
Then $\mathcal{D}=\{M \in \RMod \mid M \in \mathcal{F} \text{ and $\Omega^{-1} M \in \mathcal{C}$}\}$ is an $(n+1)$-cotilting class which is closed under injective envelopes and $\mathcal{D}_{(1)}=\mathcal{C}$.
\end{prop}

\begin{proof}
Let $C \in \ModR$ be a cotilting module such that $\mathcal{C} = {}^\perp C$ and let us denote the class from condition (4) by $\mathcal{I}$. If $\mathcal{G}$ is the Gabriel topology corresponding to $(\mathcal{T},\mathcal{F})$, we have
$$ \bigcap_{i=0}^n \Ker \ext_R^{i}(\mathcal{T},-) = \bigcap_{i=0}^n \bigcap_{I \in \mathcal{G}^f} \Ker \ext_R^{i}(R/I,-) = \bigcap_{i=0}^n \bigcap_{I \in \mathcal{G}^f} \Ker H^{i}(I;-). $$
Indeed, the first equality follows by an argument analogous to the one for Lemma~\ref{L:vanish} while the second equality follows from Corollary~\ref{C:Koszul-vanish}. Since all the $H^{i}(I;-)$ commute with direct products and direct limits, the class $\mathcal{I}$ is definable by~\cite[\S\S2.1--2.3]{CB}.

Next we claim that $\operatorname{Inj-}R \cap \mathcal{F} \subseteq \mathcal{I} \subseteq \mathcal{D}$. The first inclusion is trivial and to see the second one, let $M \in \mathcal{I}$ and consider the short exact sequence $0 \to M \to E(M) \to \Omega^{-1}(M) \to 0$. We must show that $M \in \mathcal{D}$. Clearly $M \in \mathcal{F}$ by assumption, and we also have $E(M) \in \mathcal{F}$. If $T \in \mathcal{T}$ is torsion, it follows that $\ext^i_R(T,\Omega^{-1}(M)) \cong \ext^{i+1}_R(T,M)$ for all $i\ge 0$. Thus,
$$ \Omega^{-1}(M) \in \bigcap_{i=0}^{n-1} \Ker \ext_R^{i}(\mathcal{T},-) \subseteq \mathcal{C} $$
by condition (3), and the claim is proved.

Now we prove that $\mathcal{D}$ is $(n+1)$-cotilting by checking the assumptions of Proposition~\ref{P02}. If $M \in \mathcal{D}$ and $N \subseteq M$ is a pure submodule, then $M/N \in \mathcal{C}$ since $M \in \mathcal{C}$ and definable classes are closed under pure quotients; \cite[Theorem 3.4.8]{MP}. Hence $N \in \mathcal{D}$ by Lemma~\ref{L:naive-combination}. If $(M_i)_{i \in I}$ is a directed system in $\mathcal{D}$, we fix an injective module $F \in \mathcal{F}$ which cogenerates $\mathcal{F}$ (see \cite[Proposition VI.3.7]{St}) and consider the directed system of maps $(M_i \rightarrow F_{M_i} = F^{\Hom_R(M_i,F)})_{i \in I}$ given by Lemma~\ref{functorialmap}. All the maps in the system are monomorphisms by the choice of $F$, and there is an exact sequence
$$ 0 \to \varinjlim_{i \in I} M_i \to \varinjlim_{i \in I} F_{M_i} \to \varinjlim_{i \in I} F_{M_i}/M_i \to 0. $$ 
As $F_{M_i} \in \mathcal{I}$ for each $i \in I$, we have $\varinjlim_{i \in I} F_{M_i} \in \mathcal{I} \subseteq \mathcal{D}$. Since $F_{M_i} \in \mathcal{D}$, we have $F_{M_i}/M_i \in \mathcal{C}$ for each $i \in I$, and so $\varinjlim_{i \in I} F_{M_i}/M_i \in \mathcal{C}$. It follows from Lemma~\ref{L:naive-combination} that $\varinjlim_{i \in I} M_i \in \mathcal{D}$. Finally, since $\mathcal{D}$ is closed under products by its very definition, we have verified Proposition~\ref{P02}(i).

Conditions (1) and (2) imply that $R \in \mathcal{D}$. Since $\mathcal{D}$ is definable, it must contain all projective modules as well. Hence given any $M \in \ModR$, we have $\Omega(M) \in \mathcal{D}$ if and only if $M \in \mathcal{C}$ by Lemma~\ref{L:naive-combination}. In particular $\mathcal{D}$ is closed under taking syzygies, and that $\mathcal{D}$ is closed under extensions follows from the Horseshoe Lemma. Thus Proposition~\ref{P02}(ii) holds for $\mathcal{D}$. Finally, if $M \in \ModR$, then $\Omega^n(M) \in \mathcal{C}$ since $\mathcal{C}$ is assumed to be $n$-cotilting, so $\Omega^{n+1}(M) \in \mathcal{D}$ by what we have just shown. Hence $\mathcal{D}$ is an $(n+1)$-cotilting class by Proposition~\ref{P02} and we have also proved that
$$ \mathcal{D}_{(1)} = \{ M \in \ModR \mid \Omega(M) \in \mathcal{D} \} = \mathcal{C}. $$
The closure of $\mathcal{D}$ under injective envelopes is obvious from the definition.
\end{proof}

Now we formulate an easier way to apply this result for constructing $n$-cotilting classes not of cofinite type. The constructed classes are almost identical to what we obtained for classes of cofinite type in Lemma~\ref{L02}. The only difference is that the last torsion-free class in the sequence \emph{need not} be of cofinite type (viewed as a $1$-cotilting class). 

\begin{cor} \label{cor:n-cotilt-non-cofinite}
Let $R$ be a commutative ring, let $n \ge 1$, and let $(\mathcal{F}_0,\mathcal{F}_1,\ldots,\mathcal{F}_{n-1})$ be a sequence of definable torsion-free classes such that
\begin{enumerate}
  \item[(i)] $\mathcal{F}_i$ is hereditary for each $i=0,1,\ldots,n-2$,
  \item[(ii)] $\mathcal{F}_0 \subseteq \mathcal{F}_1 \subseteq \cdots \subseteq \mathcal{F}_{n-1}$,
  \item[(iii)] $\Omega^{-i} R \in \mathcal{F}_i$ for each $i=0,1,\ldots,n-1$.
\end{enumerate}
Then $\mathcal{D}=\{M \in \ModR \mid \Omega^{-i}M \in \mathcal{F}_i \text{ for each $i=0,1,\ldots,n-1$}\}$ is an $n$-cotilting class such that $\mathcal{D}_{(i)}$ is closed under injective envelopes for each $i=0,1,\ldots,n-2$. In particular, $\mathcal{D}$ is of cofinite type if and only if the definable torsion-free class $\mathcal{F}_{n-1}$ is hereditary.
\end{cor}

\begin{proof}
We will inductively apply Proposition~\ref{P:combine} and prove the result along with the following equality:
\begin{equation} \label{E:strange-cotilt}
\mathcal{D} = \bigcap_{i=0}^{n-1} \Ker\ext^i_R(\mathcal{T}_i,-),
\end{equation}
where for each $i$, $\mathcal{T}_i$ is the torsion class corresponding to $\mathcal{F}_i$.

If $n=1$, there is nothing to prove, because a definable torsion-free class containing $R$ is 1-cotilting by Proposition~\ref{P02}. If $n>1$, we apply the inductive hypothesis to $(\mathcal{F}_1,\ldots,\mathcal{F}_{n-1})$ and obtain an $(n-1)$-cotilting class
\begin{equation} \label{E:induction}
\mathcal{C} = \bigcap_{i=0}^{n-2} \Ker\ext^{i}_R(\mathcal{T}_{i+1},-).
\end{equation}
We need to check conditions (1)--(3) of Proposition~\ref{P:combine} for $\mathcal{F}_0$ and $\mathcal{C}$ and formula~\eqref{E:strange-cotilt} for $\mathcal{D}$. However, (1) and (2) are straightforward and (3) follows immediately from~\eqref{E:induction} since $\mathcal{T}_0 \supseteq \mathcal{T}_i$ for each $i=1,\dots,n-2$. Finally, to prove~\eqref{E:strange-cotilt}, consider $M \in \mathcal{F}_0$ and a short exact sequence $0 \to M \to E(M) \to \Omega^{-1}(M) \to 0$. Then $E(M) \in \mathcal{F}_0$ and $\ext^{i}_R(T,\Omega^{-1}(M)) \cong \ext^{i+1}_R(T,M)$ for each $T \in \mathcal{T}_0$ and $i \ge 0$. It follows that in that case $M \in \mathcal{D}$ if and only if $\Omega^{-1}(M) \in \mathcal{C}$ if and only if $\ext^i_R(\mathcal{T}_i,M) = 0$ for each $i=1,\dots,n-1$. 
\end{proof}

We conclude by an explicit construction of an $n$-cotilting class not of cofinite type, by combining Corollary~\ref{cor:n-cotilt-non-cofinite} and Example~\ref{ex:example-non-cofinite}.

\begin{ex}
		\label{EX00}
		Let $n>0$ and $R$ be a local commutative ring satisfying:
		\begin{enumerate}
				\item[(i)] there is a non-trivial idempotent ideal $J$ in $R$, 
				\item[(ii)] there is a finitely generated ideal $I \subseteq J$ satisfying $\ext_R^i(R/I, R)=0$ for all $i=0,1,\ldots,n-1$.		
		\end{enumerate}
		
		First we provide a concrete example of such a ring $R$. Let $k$ be a field, and let $R$ be the ring of Puiseux series in $n$ variables. That is, 
$$R= \bigcup_{m \in \mathbb{N}} k[[x_1^{\frac{1}{m}}, x_2^{\frac{1}{m}},\ldots,x_n^{\frac{1}{m}}]]$$ 
		is the ring of formal power series in $n$ variables with exponents which are rational, but for each particular series the denominators are bounded. This ring has a unique maximal ideal $J$ consisting of all series with zero coefficient in degree 0, and $J$ is easily seen to be idempotent. Also, put $I=\Span(x_1,x_2,\ldots,x_n)$ and note that the elements $(x_1,x_2,\ldots,x_n)$ form a regular sequence. Then $K_\bullet(I)$ is a projective resolution of $R/I$, implying that $H^i(I;R)=0$ for all $i=0,1,\dots,n-1$. Therefore, we have $\ext_R^i(R/I,R)=0$ for all $i=0,1,\ldots,n-1$ by Corollary~\ref{C:Koszul-vanish}.
		
		Put $\mathcal{F}_{n-1}=\{M \in \ModR \mid M/\Soc_J(M) \in \Ker \Hom_R(R/I,-)\}$. Then $\mathcal{F}_{n-1}$ fits the construction of Example~\ref{ex:example-non-cofinite} (the Gabriel topology $\mathcal{G}$ is generated by the ideal $I$, and the choice of non-trivial idempotent ideal is $J$). This shows that $\mathcal{F}_{n-1}$ is a 1-cotilting class not of cofinite type, i.e. a definable torsion-free class containing $R$ which is not hereditary. It also follows from the assumption on $I$ that $\Ker \Hom_R(R/I,-) \subseteq \mathcal{F}_{n-1}$ contains $\Omega^{-i}R$ for all $i=0,1,\ldots,n-1$.

		We put $\mathcal{F}_{k}=\Ker \Hom_R(R/I,-)$ for all $k=0,1,\ldots,n-2$, and note that those are hereditary torsion-free classes of finite type. Then it is straightforward to check that the sequence $(\mathcal{F}_0,\mathcal{F}_1,\ldots,\mathcal{F}_{n-1})$ satisfies conditions of Corollary~\ref{cor:n-cotilt-non-cofinite}. We conclude that
		$$\mathcal{D}=\{M \in \ModR \mid \Omega^{-i}M \in \mathcal{F}_i \text{ for all $i=0,1,\ldots,n-1$}\}$$
		is an $n$-cotilting class such that $\mathcal{D}_{(i)}$ is closed under injective envelopes for all $i=0,1,\ldots,n-2$, but not for $i=n-1$. In particular, $\mathcal{D}$ is not of cofinite type.
	\end{ex}

\bibliographystyle{alpha}
\bibliography{tilt_n}

\end{document}